\documentclass[12pt,a4paper]{article}
\usepackage{amsmath,amsthm,amssymb, mathrsfs, tocloft}
\usepackage{hyperref}
\usepackage{graphicx}
\usepackage{cite}
\usepackage{pgfplots}
\usepackage{subcaption}
\usepackage{verbatim}
\graphicspath{ {./drawings/} }
\input xy
\xyoption{all}
\usepackage{epsfig}

\newtheorem{thm}{Theorem}[section]
\newtheorem{cor}[thm]{Corollary}
\newtheorem{lem}[thm]{Lemma}

 \theoremstyle{definition}
\newtheorem{defin}[thm]{Definition}

\theoremstyle{remark}
\newtheorem{remark}[thm]{Remark}

\newtheorem{example}[thm]{Example}

\numberwithin{equation}{section}

\newcommand{\ben}{\begin{enumerate}}

\newcommand{\een}{\end{enumerate}}
\newcommand{\bit}{\begin{itemize}}
\newcommand{\eit}{\end{itemize}}

\begin{document}

\title
{A Support Characterization for Functions on the Unit Sphere with Vanishing Integrals Arising from Tangent Planes to a Given Surface}

\vskip 1cm
\author{Yehonatan Salman \\ Email: salman.yehonatan@gmail.com\\ Tel-Aviv University}
\date{}

\maketitle

\begin{abstract}

Let $\Sigma$ be an axially symmetric, smooth, closed hypersurface in $\Bbb R^{n + 1}$ with a simply connected interior which is contained inside the unit sphere $\Bbb S^{n}$. For a continuous function $f$, which is defined on $\Bbb S^{n}$, the main goal of this paper is to characterize the support of $f$ in case where its integrals vanish on subspheres obtained by intersecting $\Bbb S^{n}$ with the tangent hyperplanes of a certain subdomain $\mathcal{U}\subset\Sigma$ of $\Sigma$. We show that the support of $f$ can be characterized in case where its integrals also vanish on subspheres obtained by intersecting $\Bbb S^{n}$ with hyperplanes obtained by infinitesimal perturbations of the tangent hyperplanes of $\mathcal{U}$ and where $\mathcal{U}$ satisfies some regularity condition which implies local convexity.

\end{abstract}

\section{Introduction and Motivation}

\hskip0.65cm The main mathematical tool to be investigated in this paper is the spherical transform which integrates functions, defined on the unit hypersphere $\Bbb S^{n}$ in $\Bbb R^{n + 1}$, on a prescribed $n$ dimensional family $\Omega$ of subspheres in $\Bbb S^{n}$. By definition, a subsphere in $\Bbb S^{n}$ is any nonempty subset obtained by intersecting $\Bbb S^{n}$ with a hyperplane in $\Bbb R^{n + 1}$ and thus the family $\Omega$ is completely determined by the corresponding set $\Omega'$ of hyperplanes.

Results related to the spherical transform, such as reconstruction methods, uniqueness theorems and range characterization have been derived by many authors and we do not pretend to cover here a large list of references. For a small sample of some classical and new results see \cite{2, 4, 5, 6, 7, 8, 9, 10, 11, 13, 15, 17, 18, 21, 23, 24}.

The earliest results in this subject matter were obtained by Paul Funk in \cite{4, 5} which obtained a reconstruction method for the case where $\Omega'$ consists of all hyperplanes in $\Bbb R^{3}$ which pass through the origin or equivalently where $\Omega$ consists of all great circles on $\Bbb S^{2}$, i.e., the Funk transform. In this case it was shown how the even part of a function can be reconstructed while the kernel of the Funk transform consists of odd functions. In \cite{10, 11, 24} the results above were generalized in any dimension. Similar results can be obtained for the generalized case where $\Omega'$ consists of tangent planes to a hypersphere inside $\Bbb S^{n}$ with center at the origin and radius $r, 0 \leq r < 1$, or equivalently where $\Omega$ consists of subspheres with a fixed radius, by using expansion into spherical harmonics, as was shown in \cite{5}.

Results for the similar case where $\Omega'$ consists of hyperplanes which pass through a common point on $\Bbb S^{n}$, or equivalently where $\Omega$ consists of subspheres which have a common point, were obtained in \cite{10, 11}. There, the author uses the stereographic projection in order to obtain an explicit relation between the spherical transform, which in this case is called the spherical slice transform, and the classical Radon transform in order to obtain a reconstruction method and a uniqueness theorem.

More generally, reconstruction formulae and kernel and range characterization for the case where $\Omega'$ consists of hyperplanes which pass through a fixed point inside $\Bbb S^{n}$ were obtained recently in \cite{15, 18, 19, 20, 21}.

Results for the case where $\Omega'$ consists of hyperplanes which are parallel to the last coordinate axis $x_{n + 1}$, or equivalently where $\Omega$ consists of subspheres orthogonal to the hyperplane $\mathcal{H}:x_{n + 1} = 0$, were obtained in \cite{8, 9, 23}. There it was shown how the even part of a function, with respect to the reflection through $\mathcal{H}$, can be reconstructed while the kernel consists of odd functions with respect to this reflection.\\

Observe that all the cases discussed above can be obtained as a particular case by considering sets $\Omega'$ which consist of tangent planes to a given hypersurface in $\Bbb R^{n + 1}$. Indeed, the Funk transform can be obtained by considering the set $\Omega'$ of tangent planes to a hypersphere inside $\Bbb S^{n}$ with center at the origin and radius $r$ and then taking the limit $r\rightarrow0^{+}$. Similarly, the spherical slice transform and the generalized case where $\Omega'$ consists of hyperplanes which pass through a common point $p$ inside $\Bbb S^{n}$ can be obtained exactly as before by considering hyperspheres with an arbitrary small radius and center at $p$. The case where $\Omega'$ consists of hyperplanes which are parallel to the last coordinate axis $x_{n + 1}$ can be obtained by considering a hypersphere with a center on $te_{n + 1}, t\in\Bbb R$ and radius $r > 0$ and then taking the limit $t\rightarrow\pm\infty, r\rightarrow0^{+}$.

This more generalized case where $\Omega'$ consists of tangent planes to a hypersurface $\Sigma$ was considered in \cite{14, 15, 16, 22}. The case where $\Sigma$ is a spheroid inside $\Bbb S^{n}$ with center at the origin was considered in \cite{22} and the case where $\Sigma$ is an arbitrary sphere inside $\Bbb S^{n}$ was considered in \cite{15}. The more generalized case where $\Sigma$ is an arbitrary ellipsoid inside $\Bbb S^{n}$ was considered in \cite{14, 16} where it is shown how uniqueness and reconstruction theorems can be obtained if one assumes some nontrivial assumptions on the support of each function in question.\\

In all of the results mentioned above one cannot obtain a support characterization for functions with vanishing integrals on each subsphere in the family $\Omega$. That is, one cannot guarantee that a function $f$, with vanishing integrals on each subsphere in $\Omega$, will vanish even on an arbitrary small open subset of $\Bbb S^{n}$. Indeed, any odd function is in the kernel of the Funk transform and of course one can easily choose such functions which do not vanish in any small open subset. For the more general case, where $\Omega'$ consists of tangent planes to a hypersphere inside $\Bbb S^{n}$ with center at the origin and radius $r, 0\leq r < 1$, then if $r$ is a root of the Gegenbauer polynomial of order $(n - 2) / 2$ and degree $m$ then the integrals of any spherical harmonic of degree $m$ will vanish on each subsphere in $\Omega$. Similar examples can be obtained for the case where $\Omega'$ consists of hyperplanes passing through a common point inside $\Bbb S^{n}$ (see \cite{18} where the kernel of the corresponding integral transform is characterized). For the case where $\Omega$ consists of subspheres which are orthogonal to the hyperplane $\mathcal{H}:x_{n + 1} = 0$ then the integrals of any function which is odd, with respect to the reflection through $\mathcal{H}$, will vanish on each subsphere in $\Omega$. Similarly, for the case where $\Omega'$ consists of tangent planes to an ellipsoid inside $\Bbb S^{n}$ one needs to assume some prescribed conditions on the support of each function $f$, with vanishing integrals on each subsphere of $\Omega$, in order to guarantee that $f$ will vanish on a larger subset of $\Bbb S^{n}$ (see \cite{16}).

In conclusion, one cannot expect to obtain results with regard to the support characterization by assuming that $\Omega'$ consists only of tangent planes to a given hypersurface $\Sigma$.

However, the main result of this paper asserts that if $\Sigma$ is a hypersurface, which satisfies some prescribed conditions, and $\Omega'$ consists of all tangent planes to a certain subdomain $\mathcal{U}\subset\Sigma$ of $\Sigma$, then a support characterization, for functions with vanishing integrals on each subsphere in $\Omega$, can be obtained if one allows a "slight enlargement" of $\Omega$. By this slight enlargement we mean that for the set $\Omega'$ we add all the hyperplanes which are obtained from the tangent planes to $\mathcal{U}$ by "infinitesimal perturbations". The notion of infinitesimal perturbations of a hyperplane will be defined more rigorously in the next section. Intuitively, the enlargement of $\Omega$ means that if one thinks of $\Omega$ as a manifold in the set of all subspheres in $\Bbb S^{n}$ and if the spherical transform vanishes on $\Omega$ then its first order derivatives also vanish on $\Omega$. Again, the notion of a derivative of the spherical transform will be defined more rigorously in the next section.

As for the hypersurface $\Sigma$, we assume that it closed, smooth, contained inside the unit sphere $\Bbb S^{n}$ and is axially symmetric with respect to the last coordinate axis $x_{n + 1}$. We also assume that the interior of $\Sigma$ is simply connected so for example the case where $\Sigma$ is a torus is omitted. The last condition is imposed to ensure that the last coordinate axis $x_{n + 1}$ all ways intersects $\Sigma$ which will be important to ensure the existence of the subdomain $\mathcal{U}\subset\Sigma$ as described above. For the domain $\mathcal{U}$ we assume that it satisfies an additional regularity condition which will be rigourously defined below. This regularity condition implies in particular that $\mathcal{U}$ is locally convex at each point where we say that $\mathcal{U}$ is locally convex at a point $p$ if there exists a ball $B$, with the center at $p$, such that $B\cap\mathcal{U}$ is contained on a boundary of a strictly convex domain $\mathcal{D}$, i.e., $\mathcal{D}$ is convex and its boundary does not contain any line segments.

In terms of the support characterization, assume that for the subdomain $\mathcal{U}$ of $\Sigma$, which is completely determined by $\Sigma$, the family $\Omega'$ consists of all the tangent planes to $\mathcal{U}$ together with their infinitesimal perturbations. Then, any continuous function $f$, with vanishing integrals on each subsphere in $\Omega$, has support inside the projection of $\mathcal{U}$ on the hypersphere $\Bbb S^{n}$ with respect to the north pole $e_{n + 1}$. This notion of projection will be defined more rigorously in the next section.

For the proof of the main result we exploit the relation between the spherical and spherical mean transforms which is induced by the stereographic projection. With the help of the stereographic projection we define a function which maps the subdomain $\mathcal{U}\subset\Sigma$ onto a hypersurface $\mathcal{U}'$ in $\Bbb R^{n}\times[0,\infty)$. We then show that the support characterization problem can be converted to the problem of the support characterization for functions whose spherical mean transform vanishes, together with its first order derivatives, on $\mathcal{U}'$. Our main observation is that $\mathcal{U}'$ is allways a space like surface (which will be defined in the next section) and thus we can exploit uniqueness theorems, which were obtained in \cite[Chap VI]{3}, for this type of surfaces. After characterizing the support of functions, whose spherical mean transform satisfies the assumption above, we use the inverse stereographic projection in order to obtain a support characterization for the case of the spherical transform with its corresponding family $\Omega$ of subspheres.\\

The main results of this paper are formulated in Sect. 3. Before formulating the main results we introduce some mathematical background given in Sect. 2.

\section{Mathematical Background}

\hskip0.6cm Denote by $\Bbb R^{n}$ the $n$ dimensional Euclidean space and by $\Bbb R^{+}$ the ray $[0, \infty)$. For a point $x\in\Bbb R^{n}$ and a real number $r\geq0$ denote by $\Bbb S^{n - 1}(x, r)$ the $n - 1$ dimensional sphere in $\Bbb R^{n}$ with the center at $x$ and radius $r$, i.e.,
$$\hskip-6.9cm\Bbb S^{n - 1}(x, r) = \{x'\in\Bbb R^{n}:|x' - x| = r\}.$$
In particular, we denote by $\Bbb S^{n - 1} := \Bbb S^{n - 1}(\bar{0}, 1)$ the unit sphere in $\Bbb R^{n}$. Denote by $\Bbb B_{n}(x, r)$ the open $n$ dimensional ball in $\Bbb R^{n}$ with the center at $x$ and radius $r$, i.e.,
\vskip-0.2cm
$$\hskip-7.2cm\Bbb B_{n}(x, r) = \{x'\in\Bbb R^{n}:|x' - x| < r\}.$$
In case where the dimension of $\Bbb B_{n}(x, r)$ is clear from the context we will just write $\Bbb B(x, r)$. For a point $\psi\in\Bbb S^{n}$ and a real number $\rho\geq0$ denote by $\mathrm{H}_{\psi, \rho}$ the hyperplane in $\Bbb R^{n + 1}$ with the distance $\rho$ from the origin in the direction $\psi$, i.e.,
$$\hskip-8cm\mathrm{H}_{\psi, \rho} = \{x\in\Bbb R^{n + 1}:x\cdot\psi = \rho\}.$$
In case where $\rho < 1$ we denote by $\Bbb S^{n - 1}_{\psi ,\rho}$ the $n - 1$ dimensional subsphere in $\Bbb S^{n}$ which is obtained by the intersection of $\mathrm{H}_{\psi, \rho}$ with $\Bbb S^{n}$, i.e.,
$$\hskip-6.15cm\Bbb S^{n - 1}_{\psi ,\rho} = \Bbb S^{n}\cap\mathrm{H}_{\psi, \rho} = \{x\in\Bbb S^{n}:x\cdot\psi = \rho\}.$$

For a continuous function $f$, which is defined on $\Bbb S^{n}$, define its spherical transform $\mathcal{S}f$ by
\vskip-0.2cm
$$\hskip-9.6cm(\mathcal{S}f):\Bbb S^{n}\times[0, 1)\rightarrow\Bbb R,$$
$$\hskip-8.4cm(\mathcal{S}f)(\psi, t) = \int_{\psi'\cdot\psi = t}f(\psi')dS_{\psi'}$$
where $dS_{\psi'}$ is the standard infinitesimal volume measure. For each point $(\psi, t)\in\Bbb S^{n}\times[0,1)$ the spherical transform $\mathcal{S}f$ evaluates the integral of $f$ on the $n - 1$ dimensional subsphere $\Bbb S_{\psi, t}^{n - 1}$. Sometimes, it will be more convenient to work with the unit ball $\Bbb B_{n + 1}(\overline{0},1)$ as the domain of definition rather than with the cylinder $\Bbb S^{n}\times[0,1)$ and thus we define the modified spherical transform $\mathcal{S}_{0}f$, of a function $f$, by
\vskip-0.2cm
$$\hskip-3.8cm\mathcal{S}_{0}f:\Bbb B_{n + 1}(\overline{0},1)\setminus\{\overline{0}\}\rightarrow\Bbb R, (\mathcal{S}_{0}f)(x) = (\mathcal{S}f)(x / |x|, |x|).$$
That is, for a point $x\neq\overline{0}$ inside the unit ball we consider the unique hyperplane for which $x$ is its closest point to the origin and integrate $f$ on the $n - 1$ dimensional subsphere obtained from the intersection of this hyperplane with $\Bbb S^{n}$. Observe that the modified spherical transform is not defined at $x = \overline{0}$ since this point does not determine a hyperplane in a unique way.

Similarly to the spherical transform, the spherical mean transform $\mathcal{R}f$ of a continuous function $f$, which is defined on $\Bbb R^{n}$, is defined by
$$\hskip-9.6cm(\mathcal{R}f):\Bbb R^{n}\times\Bbb R^{+}\rightarrow\Bbb R,$$
$$\hskip-8.1cm(\mathcal{R}f)(x, t) = \int_{|x' - x| = t}f(x')dS_{x'}$$
where $dS_{x'}$ is the standard infinitesimal volume measure on the sphere of integration. For a point $(x, t)\in\Bbb R^{n}\times\Bbb R^{+}$ the spherical mean transform $\mathcal{R}f$ evaluates the integral of $f$ on the hypersphere $\Bbb S^{n - 1}(x, t)$ with center at $x$ and radius $t$.

Denote by $\Lambda$ and $\Lambda^{-1}$ respectively the stereographic and inverse stereographic projections between $\Bbb S^{n}$ and $\Bbb R^{n}$:
\vskip-0.2cm
$$\hskip-9.3cm\Lambda:\Bbb S^{n}\setminus\{e_{n + 1}\}\rightarrow\Bbb R^{n},$$
$$\hskip-7.1cm\Lambda(x) = \left(\frac{x_{1}}{1 - x_{n + 1}},...,\frac{x_{n}}{1 - x_{n + 1}}\right),$$

$$\hskip-8.9cm\Lambda^{- 1}:\Bbb R^{n}\rightarrow\Bbb S^{n}\setminus\{e_{n + 1}\},$$
$$\hskip-5.15cm\Lambda^{- 1}(x) = \left(\frac{2x_{1}}{1 + |x|^{2}},...,\frac{2x_{n}}{1 + |x|^{2}},\frac{- 1 + |x|^{2}}{1 + |x|^{2}}\right).$$
\vskip0.3cm

The following important lemma will be used throughout the text. Its proof is straightforward but rather technical and hence it will be omitted. For more details see \cite[Chap. 7]{17}.

\begin{lem}

Consider the subsphere $\Bbb S^{n - 1}_{\psi, \rho}$ of $\Bbb S^{n}$ and assume that it does not pass through the north pole $e_{n + 1}$, i.e, the condition $\rho\neq\psi_{n + 1}$ is satisfied. Then, the image $\Bbb S_{\psi, \rho}^{n - 1, \ast} = \Lambda(\Bbb S^{n - 1}_{\psi, \rho})$ of $\Bbb S^{n - 1}_{\psi, \rho}$ under the stereographic projection $\Lambda$ is the $n - 1$ dimensional sphere
\begin{equation}\hskip-3.75cm\Bbb S_{\psi, \rho}^{n - 1, \ast} = \Bbb S^{n - 1}\left(\psi^{\ast} / (\rho - \psi_{n + 1}), \sqrt{1 - \rho^{2}} / |\rho - \psi_{n + 1}|\right)\end{equation}
in $\Bbb R^{n}$ with the center at $\psi^{\ast} / (\rho - \psi_{n + 1})$ and radius $\sqrt{1 - \rho^{2}} / |\rho - \psi_{n + 1}|$ where $\psi^{\ast} = (\psi_{1},...,\psi_{n})$ (i.e., $\psi^{\ast}$ is the orthogonal projection of $\psi$ to $\Bbb R^{n}$). Futhermore, let $dS$ and $dS^{\ast}$ be respectively the standard infinitesimal volume measures on $\Bbb S_{\psi, \rho}^{n - 1}$ and $\Bbb S_{\psi, \rho}^{n - 1, \ast}$. Then, for every point $x$ on $\Bbb S_{\psi ,\rho}^{n - 1, \ast}$ we have the following relation:
\vskip-0.2cm
$$\hskip-6cm dS(\Lambda^{-1}(x)) = 2^{n - 1}dS^{\ast}(x)/(1 + |x|^{2})^{n - 1}.$$
\end{lem}

From the last equation and equation (2.1) in Lemma 2.1 it follows that we have the following relation between the spherical and spherical mean transforms:
\vskip-0.2cm
\begin{equation}\hskip-2.75cm(\mathcal{S}f)(\psi, \rho) = (\mathcal{R}g)\left(\psi^{\ast} / (\rho - \psi_{n + 1}), \sqrt{1 - \rho^{2}} / |\rho - \psi_{n + 1}|\right)\end{equation}
where
\vskip-0.2cm
\begin{equation}\hskip-4.75cm g(x) = 2^{n - 1}f\left(\Lambda^{-1}(x)\right)/(1 + |x|^{2})^{n - 1}, x\in\Bbb R^{n}.\end{equation}

\begin{defin}

A \textbf{hypersurface} $\Sigma$ in $\Bbb R^{n + 1}$ is an $n$ dimensional set which is also the set of solutions to an equation of the form $F(x) = 0$ for a continuous function $F:\Bbb R^{n + 1}\rightarrow\Bbb R$. If $F$ is also continuously differentiable and $\nabla F(x)\neq\bar{0}$ whenever $F(x) = 0$ then $\Sigma$ is called \textbf{smooth}. A one dimensional hypersurface is called a \textbf{curve}.

\end{defin}

\begin{defin}

Let $\Sigma$ be a curve in $\Bbb R^{2}$ and let $\Sigma^{\ast}$ be a subcurve of $\Sigma$ (i.e., a connected subset). Then, $\Sigma^{\ast}$ is called regular if there exists an interval $I$ and a twice continuously differential path function $\gamma:I\rightarrow\Bbb R^{2}$ which parameterizes $\Sigma^{\ast}$ and which satisfies the condition $\gamma_{1}'(t)\gamma_{2}''(t) - \gamma_{1}''(t)\gamma_{2}'(t)\neq0, t\in I$.

\end{defin}

\begin{defin}

For a smooth hypersurface $\Sigma$ denote by $\mathrm{H}_{x}(\Sigma)$ its tangent hyperplane at the point $x\in\Sigma$. Then, the \textbf{singularity set} $\Sigma'\subset\Sigma$ of $\Sigma$ is the set of points $x$ on $\Sigma$ for which the tangent plane $\mathrm{H}_{x}(\Sigma)$ passes through the north pole $e_{n + 1}$ of the unit sphere $\Bbb S^{n}$. We also define the larger subset $\Sigma_{0}\subset\Sigma$ as the set of points on $\Sigma$ whose tangent planes pass through the north pole $e_{n + 1}$ or through the origin.

\end{defin}

\begin{defin}

A closed hypersurface $\Sigma\subset\Bbb R^{n + 1}$ is said to be \textbf{axially symmetric} if it is invariant with respect to rotations which fix the last coordinate axis $x_{n + 1}$. A connected axially symmetric subdomain $\mathcal{U}\subseteq\Sigma$, of an axially symmetric closed hypersurface $\Sigma$, is called \textbf{regular} if it is obtained as the surface of revolution of a subcurve which is regular.

\end{defin}

We now present the concept of space like surfaces which was introduced in \cite{3}. For more details and results on this type of surfaces see \cite[Chap VI]{3}.

\begin{defin}

A smooth, connected hypersurface $\Sigma$ in $\Bbb R^{n + 1}$ is called \textbf{space like} if at each point $x\in\Sigma$ the normal $\mathbf{N}(x)$, of $\Sigma$ at $x$, satisfies $\mathbf{N}_{n + 1}^{2}(x)\geq\sum_{i = 1}^{n}\mathbf{N}_{i}^{2}(x)$. That is, at each point $x\in\Sigma$ the normal $\mathbf{N}(x)$ is contained in the closure of the interior of the right circular cone with apex at $x$.

\end{defin}

\begin{remark}

Geometrically, space like surfaces can be described as surfaces which cannot "approach infinity" faster than a cone. More accurately, if $\Sigma$ is space like then at each point $x\in\Sigma$, the upper part of the interior of the right circular cone with apex at $x$ is contained above $\Sigma$. Similarly, the lower part of the interior of this cone is contained below $\Sigma$. We will use this fact in the proof of the main result.

\end{remark}


\begin{figure}

\caption{\small{For the hypersurface $\Sigma\subset\Bbb R^{3}$ in the left image below the red, blue, yellow and green areas are the connected components of $\Sigma\setminus\Sigma'$ where $\Sigma'$ is the singularity set of $\Sigma$. A 2D section of this image is shown in the right image. Observe that the tangent planes, at the boundary points of these connected components, pass through the north pole}.}
\vskip0.5cm

\hskip-1cm\begin{tikzpicture}

\begin{axis}[view={30}{10}]
\addplot3 [
surf, red, z buffer=sort,
samples=20,domain = -2.95 / 2:-0.2, y domain = 0:2 * pi,
] (
{(1 / 2 + (1 / 5) * sin(deg(3 * x + 3.1))) * (cos(deg(x))) * cos(deg(y))},
{(1 / 2 + (1 / 5) * sin(deg(3 * x + 3.1))) * (cos(deg(x))) * sin(deg(y))},
{(1 / 2 + (1 / 5) * sin(deg(3 * x + 3.1))) * (sin(deg(x))) - 1 / 5}
);

\addplot3 [
surf, blue, z buffer=sort,
samples=20,domain = -0.2:0.6, y domain = 0:2 * pi,
] (
{(1 / 2 + (1 / 5) * sin(deg(3 * x + 3.1))) * (cos(deg(x))) * cos(deg(y))},
{(1 / 2 + (1 / 5) * sin(deg(3 * x + 3.1))) * (cos(deg(x))) * sin(deg(y))},
{(1 / 2 + (1 / 5) * sin(deg(3 * x + 3.1))) * (sin(deg(x))) - 1 / 5}
);

\addplot3 [
surf, yellow, z buffer=sort,
samples=20,domain = 0.6:1.1, y domain = 0:2 * pi,
] (
{(1 / 2 + (1 / 5) * sin(deg(3 * x + 3.1))) * (cos(deg(x))) * cos(deg(y))},
{(1 / 2 + (1 / 5) * sin(deg(3 * x + 3.1))) * (cos(deg(x))) * sin(deg(y))},
{(1 / 2 + (1 / 5) * sin(deg(3 * x + 3.1))) * (sin(deg(x))) - 1 / 5}
);

\addplot3 [
surf, green, z buffer=sort,
samples=20,domain = 1.1:2.06, y domain = 0:2 * pi,
] (
{(1 / 2 + (1 / 5) * sin(deg(3 * x + 3.1))) * (cos(deg(x))) * cos(deg(y))},
{(1 / 2 + (1 / 5) * sin(deg(3 * x + 3.1))) * (cos(deg(x))) * sin(deg(y))},
{(1 / 2 + (1 / 5) * sin(deg(3 * x + 3.1))) * (sin(deg(x))) - 1 / 5}
);

\end{axis}

\end{tikzpicture}\begin{tikzpicture}

\begin{axis}

\addplot [black, samples = 50, domain = -pi:pi] ({cos(deg(x))}, {sin(deg(x))});

\draw [fill] (1000, 2000) circle (0.05cm);

\addplot [red, samples = 50, domain = -3.2 / 2 : -0.2]
(
    {(0.5 + 0.2 * sin(deg(3 * x + 3.1))) * cos(deg(x))},
    {(0.5 + 0.2 * sin(deg(3 * x + 3.1))) * sin(deg(x)) - 0.2}
);

\addplot [red, samples = 50, domain = -3.2 / 2 : -0.2]
(
    {- (0.5 + 0.2 * sin(deg(3 * x + 3.1))) * cos(deg(x))},
    {(0.5 + 0.2 * sin(deg(3 * x + 3.1))) * sin(deg(x)) - 0.2}
);

\addplot [blue, samples = 50, domain = -0.2 : 0.6]
(
    {(0.5 + 0.2 * sin(deg(3 * x + 3.1))) * cos(deg(x))},
    {(0.5 + 0.2 * sin(deg(3 * x + 3.1))) * sin(deg(x)) - 0.2}
);

\addplot [blue, samples = 50, domain = -0.2 : 0.6]
(
    {- (0.5 + 0.2 * sin(deg(3 * x + 3.1))) * cos(deg(x))},
    {(0.5 + 0.2 * sin(deg(3 * x + 3.1))) * sin(deg(x)) - 0.2}
);

\addplot [orange, samples = 50, domain = 0.6 : 1.1]
(
    {(0.5 + 0.2 * sin(deg(3 * x + 3.1))) * cos(deg(x))},
    {(0.5 + 0.2 * sin(deg(3 * x + 3.1))) * sin(deg(x)) - 0.2}
);

\addplot [orange, samples = 50, domain = 0.6 : 1.1]
(
    {- (0.5 + 0.2 * sin(deg(3 * x + 3.1))) * cos(deg(x))},
    {(0.5 + 0.2 * sin(deg(3 * x + 3.1))) * sin(deg(x)) - 0.2}
);

\addplot [green, samples = 50, domain = 1.1 : 2.06]
(
    {(0.5 + 0.2 * sin(deg(3 * x + 3.1))) * cos(deg(x))},
    {(0.5 + 0.2 * sin(deg(3 * x + 3.1))) * sin(deg(x)) - 0.2}
);

\addplot [green, samples = 50, domain = 1.1 : 2.06]
(
    {- (0.5 + 0.2 * sin(deg(3 * x + 3.1))) * cos(deg(x))},
    {(0.5 + 0.2 * sin(deg(3 * x + 3.1))) * sin(deg(x)) - 0.2}
);

\addplot [black, samples = 50, domain = -0.25 : 0]({x},{1 + 2.8 * x});

\addplot [black, samples = 50, domain = -0.25 : 0]({-x},{1 + 2.8 * x});

\addplot [black, samples = 50, domain = -0.6 : 0]({x},{1 + 2.16 * x});

\addplot [black, samples = 50, domain = -0.6 : 0]({-x},{1 + 2.16 * x});

\addplot [black, samples = 50, domain = -0.25 : 0]({x},{1 + 4.1 * x});

\addplot [black, samples = 50, domain = -0.25 : 0]({-x},{1 + 4.1 * x});

\end{axis}

\begin{axis}

\addplot [black, samples = 50, domain = -pi:pi] ({cos(deg(x))}, {sin(deg(x))});

\draw [fill] (1000, 2000) circle (0.05cm);

\addplot [red, samples = 50, domain = -3.2 / 2 : -0.2]
(
    {(0.5 + 0.2 * sin(deg(3 * x + 3.1))) * cos(deg(x))},
    {(0.5 + 0.2 * sin(deg(3 * x + 3.1))) * sin(deg(x)) - 0.2}
);

\addplot [red, samples = 50, domain = -3.2 / 2 : -0.2]
(
    {- (0.5 + 0.2 * sin(deg(3 * x + 3.1))) * cos(deg(x))},
    {(0.5 + 0.2 * sin(deg(3 * x + 3.1))) * sin(deg(x)) - 0.2}
);

\addplot [blue, samples = 50, domain = -0.2 : 0.6]
(
    {(0.5 + 0.2 * sin(deg(3 * x + 3.1))) * cos(deg(x))},
    {(0.5 + 0.2 * sin(deg(3 * x + 3.1))) * sin(deg(x)) - 0.2}
);

\addplot [blue, samples = 50, domain = -0.2 : 0.6]
(
    {- (0.5 + 0.2 * sin(deg(3 * x + 3.1))) * cos(deg(x))},
    {(0.5 + 0.2 * sin(deg(3 * x + 3.1))) * sin(deg(x)) - 0.2}
);

\addplot [orange, samples = 50, domain = 0.6 : 1.1]
(
    {(0.5 + 0.2 * sin(deg(3 * x + 3.1))) * cos(deg(x))},
    {(0.5 + 0.2 * sin(deg(3 * x + 3.1))) * sin(deg(x)) - 0.2}
);

\addplot [orange, samples = 50, domain = 0.6 : 1.1]
(
    {- (0.5 + 0.2 * sin(deg(3 * x + 3.1))) * cos(deg(x))},
    {(0.5 + 0.2 * sin(deg(3 * x + 3.1))) * sin(deg(x)) - 0.2}
);

\addplot [green, samples = 50, domain = 1.1 : 2.06]
(
    {(0.5 + 0.2 * sin(deg(3 * x + 3.1))) * cos(deg(x))},
    {(0.5 + 0.2 * sin(deg(3 * x + 3.1))) * sin(deg(x)) - 0.2}
);

\addplot [green, samples = 50, domain = 1.1 : 2.06]
(
    {- (0.5 + 0.2 * sin(deg(3 * x + 3.1))) * cos(deg(x))},
    {(0.5 + 0.2 * sin(deg(3 * x + 3.1))) * sin(deg(x)) - 0.2}
);

\addplot [black, samples = 50, domain = -0.25 : 0]({x},{1 + 2.8 * x});

\addplot [black, samples = 50, domain = -0.25 : 0]({-x},{1 + 2.8 * x});

\addplot [black, samples = 50, domain = -0.6 : 0]({x},{1 + 2.16 * x});

\addplot [black, samples = 50, domain = -0.6 : 0]({-x},{1 + 2.16 * x});

\addplot [black, samples = 50, domain = -0.25 : 0]({x},{1 + 4.1 * x});

\addplot [black, samples = 50, domain = -0.25 : 0]({-x},{1 + 4.1 * x});

\end{axis}

\end{tikzpicture}

\end{figure}
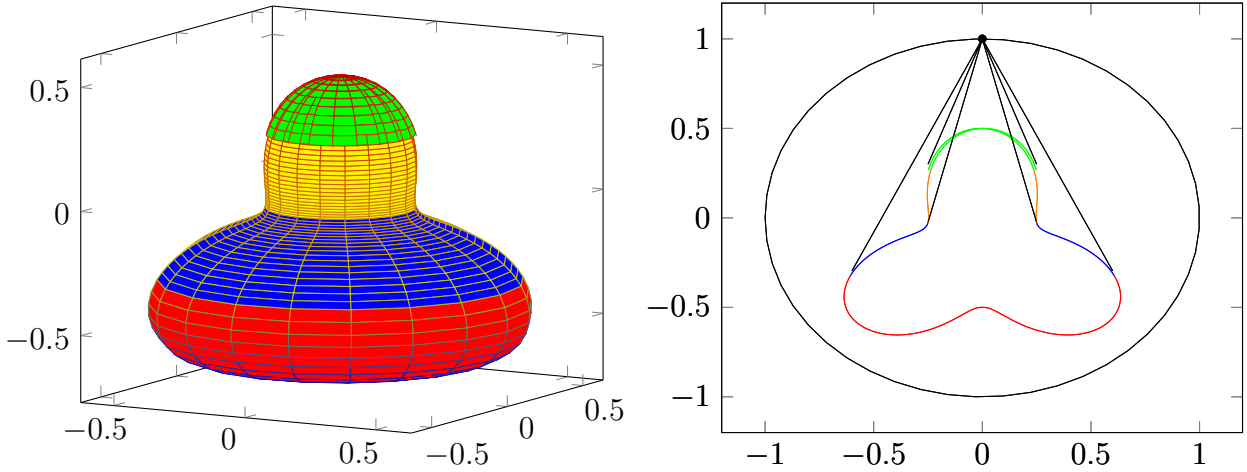


\vskip3cm
\hskip-0.6cm Let
\vskip-0.2cm
$$\hskip-7.8cm\Bbb S_{0} = \Bbb S^{n}\left(e_{n + 1} / 2, 1 / 2\right)\subset \overline{\Bbb B(0,1)}$$
be the $n$ dimensional sphere with the center at $e_{n + 1} / 2 = (0,...,0,1/2)$ and radius $1 / 2$. For a closed, smooth hypersurface $\Sigma$ in $\Bbb R^{n + 1}$, which is contained inside the unit sphere $\Bbb S^{n}$, define the map
\vskip-0.2cm
$$\hskip-8.85cm\Psi_{0}:\Sigma\setminus\Sigma_{0}\rightarrow\Bbb B(0,1)\setminus\Bbb S_{0}$$
as follows: for a point $x$ in $\Sigma\setminus\Sigma_{0}$ let $\mathrm{H}_{x}(\Sigma)$ be the tangent plane of $\Sigma$ at $x$ and let $x'$ be its closest point to the origin, then $\Psi_{0}(x) = x'$. Equivalently, if $\mathrm{H}_{x}(\Sigma)$ has unit normal $\psi$ and distance $\rho\geq0$ from the origin (the sign of $\psi$ is chosen so that $\rho\psi\in\mathrm{H}_{x}(\Sigma)$) then $\Psi_{0}(x) = \rho\psi$.

If $\Sigma$ is given by the equation $F(x) = 0$ then $\Psi_{0}$ is given explicitly by
\vskip-0.2cm
$$\hskip-6cm\Psi_{0}(x) = \left[\frac{x\cdot\nabla F(x)}{|\nabla F(x)|^{2}}\right]\nabla F(x), x\in\Sigma\setminus\Sigma_{0}.$$
Observe that if $\mathrm{H}_{x}(\Sigma)$ is a tangent plane to $\Sigma$ at a point $x\in\Sigma_{0}$ then, if $\mathrm{H}_{x}(\Sigma)$ passes through the north pole $e_{n + 1}$ then the image of $x$ under $\Psi_{0}$ will be a point in $\Bbb S_{0}$ while if $\mathrm{H}_{x}(\Sigma)$ passes through the origin then $x$ will be mapped to the origin. Hence, $\Bbb S_{0}$ is omitted from the range of $\Psi_{0}$. Of course, $\Psi_{0}$ can be defined on the whole of $\Sigma$. However, later for the proof of the main result we will have to compose $\Psi_{0}$ with a map which is not defined on $\Bbb S_{0}$. Hence, we need to reduce the domain of definition of $\Psi_{0}$ accordingly.


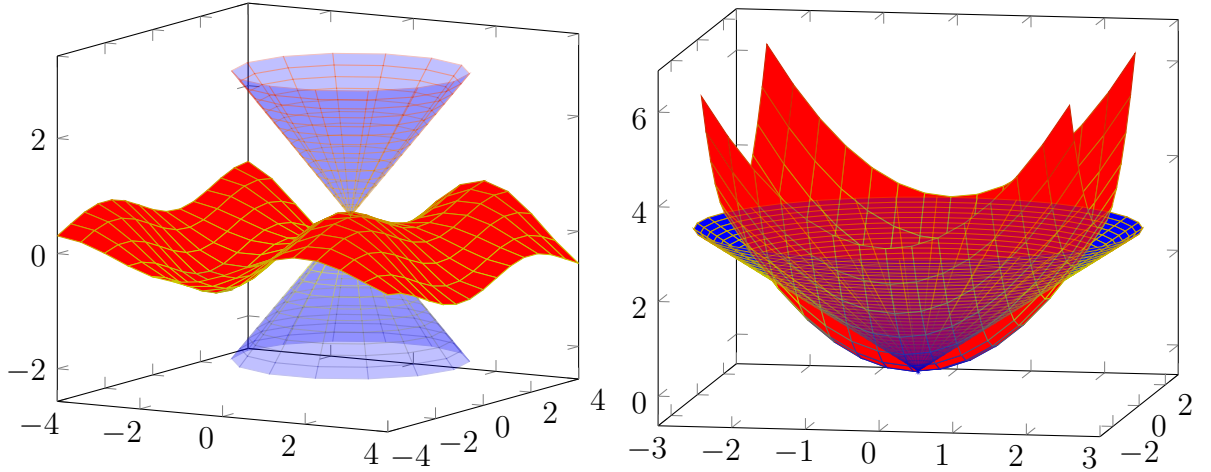
\begin{figure}

\caption{\small{In left and right images we give two examples of hypersurfaces $\Gamma_{1}$ and  $\Gamma_{2}$ respectively where $\Gamma_{1}$ is space-like while $\Gamma_{2}$ is not. Observe that the upper and lower parts of the right circular cone, which is emanating from a point on $\Gamma_{1}$, are respectively above and below $\Gamma_{1}$ while the upper part of the right circular cone, which is emanating from a point on $\Gamma_{2}$, intersects $\Gamma_{2}$}.}
\vskip0.5cm

\begin{tikzpicture}

\begin{axis}[view={30}{10}]

\addplot3 [
surf, blue, z buffer=sort, opacity = 0.25,
samples=15,domain = -2.5:0, y domain = -180:180,
] (
{0.5 + x * cos(y)},
{0.5 + x * sin(y)},
{sin(120) / 2 + x}
);

\addplot3 [
surf, red, z buffer=sort,
samples=15,domain = -4:4, y domain = -4:4,
] (
{x},
{y},
{sin(60 * x) / 2 + cos(60 * (x + y)) / 4}
);

\addplot3 [
surf, blue, z buffer=sort, opacity = 0.25,
samples=15,domain = 0:2.5, y domain = -180:180,
] (
{0.5 + x * cos(y)},
{0.5 + x * sin(y)},
{sin(120) / 2 + x}
);

\end{axis}

\end{tikzpicture}\begin{tikzpicture}

\begin{axis}[view={10}{10}]

\addplot3 [
surf, blue, z buffer=sort, opacity = 1,
samples=35,domain = 0:3, y domain = -180:180,
] (
{x * cos(y)},
{x * sin(y)},
{x}
);

\addplot3 [
surf, red, z buffer=sort,
samples=15,domain = -2.5:2.5, y domain = -2.5:2.5,
] (
{x},
{y},
{x * x / 2 + y * y / 2}
);

\addplot3 [
surf, blue, z buffer=sort, opacity = 0.25,
samples=35,domain = 0:3, y domain = -180:180,
] (
{x * cos(y)},
{x * sin(y)},
{x}
);

\end{axis}

\end{tikzpicture}

\end{figure}


Now we are ready to define the concept of vanishing of the spherical transform on a family of subspheres with their infinitesimal perturbations. For abbreviation, we denote by $\Bbb S_{x}(\Sigma)$ the subsphere of $\Bbb S^{n}$ obtained by intersection of $\Bbb S^{n}$ with the tangent plane $\mathrm{H}_{x}(\Sigma)$ of the hypersurface $\Sigma$ at the point $x$.

\begin{defin}

Let $\Sigma$ be a closed, smooth hypersurface in $\Bbb R^{n + 1}$, which is contained inside $\Bbb S^{n}$, let $\Delta$ be a subset of $\Sigma\setminus\Sigma'$ and let $f$ be a continuous function defined on $\Bbb S^{n}$. Then, we say that the spherical transform $\mathcal{S}f$ of $f$ vanishes on the set of subspheres $\Bbb S_{x}(\Sigma), x\in \Delta$, together with their infinitesimal perturbations, if
\begin{equation}\hskip-2cm(\mathcal{S}_{0}f)(\Psi_{0}(x)) = 0, \partial_{i}(\mathcal{S}_{0}f)(\Psi_{0}(x)) = 0, i = 1,..., n + 1, \forall x\in \Delta\setminus\Sigma_{0}.\end{equation}

\end{defin}

Hence, for the notion of infinitesimal perturbations we first identify the family $\Omega$ of subpheres of integration with a submanifold $\mathbf{M}\subset\Bbb B(\bar{0},1)$ with the help of the function $\Psi_{0}$. Then, we say that the spherical transform vanishes on subspheres in $\Omega$ with their infinitesimal perturbations if the modified spherical transform vanishes on $\mathbf{M}$ with its first order derivatives.

Finally, we define the concept of the tangent cone and its projection set for axially symmetric surfaces. First, observe that if the smooth hypersurface $\Sigma\subset\Bbb R^{n + 1}$ is axially symmetric and has a simply connected interior then the last coordinate axis $x_{n + 1}$ intersects $\Sigma$ at exactly two points $p^{+}$ and $p^{-}$ where $p^{+}\cdot e_{n + 1} > p^{-}\cdot e_{n + 1}$. Let $\mathcal{U}$ be the component of $\Sigma\setminus\Sigma'$ which contains the point $p^{+}$, we will call $\mathcal{U}$ the \textbf{upper connected component} of $\Sigma\setminus\Sigma'$. We define the \textbf{tangent cone} $C_{\Sigma}$ of $\Sigma$ as the union of lines which pass through the north pole $e_{n + 1}$ and a boundary point of $\mathcal{U}$. Observe that $\Bbb S^{n}\setminus (C_{\Sigma}\setminus\{e_{n + 1}\})$ consists of two connected components and we denote by $\Pi_{\Sigma}\subseteq\Bbb S^{n}$ the "lower component"  which contains the south pole $-e_{n + 1}$, i.e., the intersection of the interior of the cone $C_{\Sigma}$ with $\Bbb S^{n}$. We then say that $\Pi_{\Sigma}$ is the \textbf{projected set of the cone} $C_{\Sigma}$ on $\Bbb S^{n}$ generated by the hypersurface $\Sigma$.

\section{The Main Results}

The main result of this paper is given in Theorem 3.1 below. To make the formulation less cumbersome let us denote by $\mathbf{\Sigma}$ the set of axially symmetric, closed, smooth hypersurfaces with a simply connected interior and which are contained inside $\Bbb S^{n}$. Let us also denote by $C_{0}(\Bbb S^{n})$ the set of continuous functions on $\Bbb S^{n}$ which are compactly supported with respect to the north pole, i.e., $f$ is in $C_{0}(\Bbb S^{n})$ if and only if $f$ is continuous and vanishes in a neighborhood of $e_{n + 1}$.

\begin{thm}

Let $\Sigma$ be a hypersurface in $\mathbf{\Sigma}$ and let $\mathcal{U}$ be the upper connected component of $\Sigma\setminus\Sigma'$ and assume that it is also regular. Let $Sf$ be the spherical transform of a function $f$ in $C_{0}(\Bbb S^{n})$ and suppose that it vanishes on the family of subspheres $\Bbb S_{x}(\Sigma), x\in\mathcal{U}$, together with their infinitesimal perturbations. Then, $f$ is supported in the projection set $\Pi_{\Sigma}$ of the cone $C_{\Sigma}$ on $\Bbb S^{n}$ generated by the hypersurface $\Sigma$.

\end{thm}

The method in the proof of Theorem 3.1 is based on the following insight about the manifold of hyperspheres in $\Bbb R^{n}$ which is obtained as the set of images, under the stereographic projection, of the family of subspheres $\Bbb S_{x}(\Sigma), x\in\mathcal{U}$ where $\mathcal{U}$ is the upper connected component of $\Sigma\setminus\Sigma'$. Observe that every hypersphere in $\Bbb R^{n}$ with center at $x$ and radius $t\geq0$ can be identified with the point $(x, t)$ in $\Bbb R^{n}\times\Bbb R^{+}$. Our main observation is that if every hypersphere in $\Bbb R^{n}$, which is obtained as an image, under the stereographic projection, of a subsphere $\Bbb S_{x}(\Sigma)$, for some $x\in\mathcal{U}$, is identified with its corresponding point in $\Bbb R^{n}\times\Bbb R^{+}$ then the union of these points is a space like surface. Since this result is important by itself it will be formulated in Theorem 3.6 below as the second main result of this paper.

Results concerning uniqueness theorems and support characterization for functions satisfying second degree hyperbolic partial differential equations with initial data on a space like surface have been obtained in \cite[Chapter VI]{3}. Using the main observation as discussed above, the connection between the spherical and spherical mean transforms (2.2) and the fact that the spherical mean transform satisfies the Darboux's equation we obtain, using the results obtained in \cite{3}, a support characterization for the function $g$ where the relation between $f$ and $g$ is given by (2.3). Finally, using the inverse stereographic projection and the support characterization obtained for $g$ we will be able to characterize the support of $f$.

Before formulating Theorem 3.6 we introduce the concept of the surface map. Let us first define the following function

$$\hskip-5.35cm\Psi:\Bbb S^{n}\times[0,1)\setminus\{(\psi, \rho):\rho = \psi_{n + 1}\}\rightarrow\Bbb R^{n}\times\Bbb R^{+},$$
\begin{equation}\hskip-5.1cm\Psi(\psi, \rho) = \left(\psi^{\ast} / (\rho - \psi_{n + 1}), \sqrt{1 - \rho^{2}} / |\rho - \psi_{n + 1}|\right).\end{equation}
Let us explain on a geometrical level how $\Psi$ operates. For a point $(\psi, \rho)$ in the domain of $\Psi$ we consider the subsphere $\Bbb S_{\psi, \rho}^{n - 1}$ of $\Bbb S^{n}$ which does not pass through the north pole $e_{n + 1}$ because of the condition $\rho\neq\psi_{n + 1}$. The image of this subsphere, under the stereographic projection, is a hypersphere in $\Bbb R^{n}$ and from Lemma 2.1 its center and radius respectively are given as in the right hand side of (3.1). Finally, we identify this hypersphere with its corresponding point in $\Bbb R^{n}\times\Bbb R^{+}$.

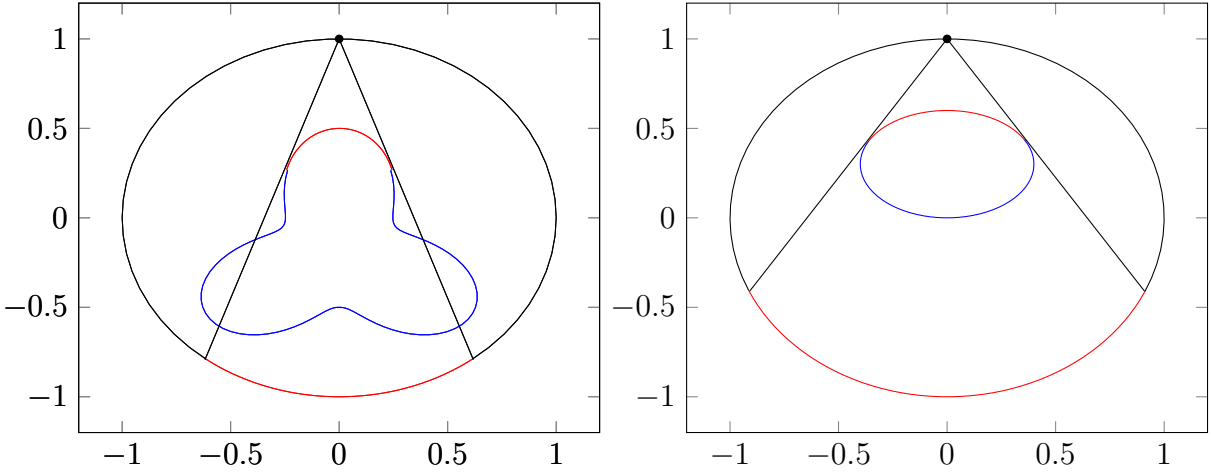
\begin{figure}

\caption{\small{In the two images below the red part on the surface $\Sigma$ is the upper connected component
of $\Sigma\setminus\Sigma'$. The two black segments which are tangent to $\Sigma$ are the cone $C_{\Sigma}$ and the red part on the unit circle is the projection set $\Pi_{\Sigma}$. Of course, these images are two dimensional sections of the corresponding surfaces in higher dimensions.}}
\vskip0.5cm

\hskip-1cm\begin{tikzpicture}

\begin{axis}

\addplot [blue, samples = 50, domain = -3.2 / 2 : 1.1]
(
    {(0.5 + 0.2 * sin(deg(3 * x + 3.1))) * cos(deg(x))},
    {(0.5 + 0.2 * sin(deg(3 * x + 3.1))) * sin(deg(x)) - 0.2}
);

\addplot [blue, samples = 50, domain = -3.2 / 2 : 1.1]
(
    {- (0.5 + 0.2 * sin(deg(3 * x + 3.1))) * cos(deg(x))},
    {(0.5 + 0.2 * sin(deg(3 * x + 3.1))) * sin(deg(x)) - 0.2}
);

\addplot [red, samples = 50, domain = 1.1 : 2.06]
(
    {- (0.5 + 0.2 * sin(deg(3 * x + 3.1))) * cos(deg(x)) - 0.005},
    {(0.5 + 0.2 * sin(deg(3 * x + 3.1))) * sin(deg(x)) - 0.2}
);

\addplot [black, samples = 50, domain = -pi / 4 - 0.125: 5 * pi / 4 + 0.125]
(
    {cos(deg(x))}, {sin(deg(x))}
);

\addplot [red, samples = 50, domain = 5 * pi / 4 + 0.125: 7 * pi / 4 - 0.125]
(
    {cos(deg(x))}, {sin(deg(x))}
);

\addplot [black, samples = 50, domain = -0.618 : 0]({x},{1 + 2.9 * x});

\addplot [black, samples = 50, domain = -0.618 : 0]({-x},{1 + 2.9 * x});

\end{axis}

\begin{axis}

\addplot [blue, samples = 50, domain = -3.2 / 2 : 1.1]
(
    {(0.5 + 0.2 * sin(deg(3 * x + 3.1))) * cos(deg(x))},
    {(0.5 + 0.2 * sin(deg(3 * x + 3.1))) * sin(deg(x)) - 0.2}
);

\addplot [blue, samples = 50, domain = -3.2 / 2 : 1.1]
(
    {- (0.5 + 0.2 * sin(deg(3 * x + 3.1))) * cos(deg(x))},
    {(0.5 + 0.2 * sin(deg(3 * x + 3.1))) * sin(deg(x)) - 0.2}
);

\addplot [red, samples = 50, domain = 1.1 : 2.06]
(
    {- (0.5 + 0.2 * sin(deg(3 * x + 3.1))) * cos(deg(x)) - 0.005},
    {(0.5 + 0.2 * sin(deg(3 * x + 3.1))) * sin(deg(x)) - 0.2}
);

\addplot [black, samples = 50, domain = -pi / 4 - 0.125: 5 * pi / 4 + 0.125]
(
    {cos(deg(x))}, {sin(deg(x))}
);

\addplot [red, samples = 50, domain = 5 * pi / 4 + 0.125: 7 * pi / 4 - 0.125]
(
    {cos(deg(x))}, {sin(deg(x))}
);

\draw [fill] (1000, 2000) circle (0.05cm);

\addplot [black, samples = 50, domain = -0.618 : 0]({x},{1 + 2.9 * x});

\addplot [black, samples = 50, domain = -0.618 : 0]({-x},{1 + 2.9 * x});

\end{axis}

\end{tikzpicture}\begin{tikzpicture}

\begin{axis}

\addplot [black, samples = 50, domain = -0.425: 3.56]
(
    {cos(deg(x))}, {sin(deg(x))}
);

\addplot [red, samples = 50, domain = 3.56: -0.425 + 2 * pi]
(
    {cos(deg(x))}, {sin(deg(x))}
);

\addplot [red, samples = 50, domain = 0.45: 2.7]
(
    {0.4 * cos(deg(x))}, {0.3 * sin(deg(x)) + 0.3}
);

\addplot [blue, samples = 50, domain = 2.7: 6.75]
(
    {0.4 * cos(deg(x))}, {0.3 * sin(deg(x)) + 0.3}
);

\addplot [black, samples = 50, domain = -0.91 : 0]({x},{1 + 1.55 * x});

\addplot [black, samples = 50, domain = -0.91 : 0]({- x},{1 + 1.55 * x});

\draw [fill] (1000, 2000) circle (0.05cm);

\end{axis}

\end{tikzpicture}

\end{figure}

\begin{defin}

For a closed, smooth hypersurface $\Sigma$ in $\Bbb R^{n + 1}$, which is contained inside $\Bbb S^{n}$, define its \textbf{surface map} $\Phi_{\Sigma}$ by
$$\hskip-6.55cm\Phi_{\Sigma}:\Sigma\setminus\Sigma'\rightarrow\Bbb R^{n + 1}, \Phi_{\Sigma}(x) = \Psi\left(\psi_{x}, \rho_{x}\right)$$
where $\psi_{x}\in\Bbb S^{n}$ and $\rho_{x}\geq0$ are respectively the normal and distance, from the origin, of the tangent plane $\mathrm{H}_{x}(\Sigma)$ of $\Sigma$ at $x$ (the sign of $\psi_{x}$ is chosen such that $\rho_{x}\psi_{x}\in\mathrm{H}_{x}(\Sigma)$ and in case where $\rho_{x} = 0$ one can choose any sign for $\psi_{x}$).

\end{defin}

\begin{remark}

Similarly to the function $\Psi$ the surface map $\Phi_{\Sigma}$ operates as follows. For each point $x$ in $\Sigma$ we consider its tangent plane $\mathrm{H}_{x}(\Sigma)$ and then intersect it with the unit sphere $\Bbb S^{n}$. Observe that since $\Sigma$ is contained inside $\Bbb S^{n}$ this intersection in never empty. This intersection is a subsphere in $\Bbb S^{n}$ which is then projected to a hypersphere in $\Bbb R^{n}$ by the stereographic projection $\Lambda$. Finally, the projected hypersphere is being identified with a point in $\Bbb R^{n + 1}$ corresponding to its center and radius respectively.

\end{remark}

\begin{remark}

Observe that for a closed, smooth hypersurface $\Sigma$ in $\Bbb R^{n + 1}$, which is contained inside $\Bbb S^{n}$, $\Phi_{\Sigma}$ is defined on $\Sigma\setminus\Sigma'$ rather than on the whole of $\Sigma$. This is because on $\Sigma'$ the tangent hyperplanes pass through the north pole $e_{n + 1}$ and thus their intersections with $\Bbb S^{n}$ will be spheres that also pass through $e_{n + 1}$. However, in this case the stereographic projection $\Lambda$ will map each such sphere to a hyperplane in $\Bbb R^{n}$ rather than to a hypersphere.

\end{remark}

\begin{remark}

Let $f$ and $g$ be continuously differentiable functions, defined on $\Bbb S^{n}$ and $\Bbb R^{n}$ respectively, which are related by equation (2.3) and let $\Sigma$ be a closed, smooth hypersurface which is contained inside $\Bbb S^{n}$. Then, the spherical transform $\mathcal{S}f$ of $f$ is given on the set of subspheres obtained by intersections of tangent planes to $\Sigma\setminus\Sigma'$ with $\Bbb S^{n}$ if and only if the spherical mean transform $\mathcal{R}g$ of $g$ is given on the image $\Phi_{\Sigma}(\Sigma\setminus\Sigma')$ of $\Sigma\setminus\Sigma'$ under $\Phi_{\Sigma}$.

\end{remark}

The second main result of this paper is the following.

\begin{thm}

Let $\Sigma$ be a hypersurface in $\mathbf{\Sigma}$ and let $\mathcal{V}$ be a connected component of $\Sigma\setminus\Sigma'$ ($\hskip0.05cm\Sigma'$ is the set of singularity of $\Sigma$) which is simply connected and regular. Then, the image $\Phi_{\Sigma}(\mathcal{V})$ of $\mathcal{V}$ under $\Phi_{\Sigma}$ is a space like surface.

\end{thm}

Observe that if $\Sigma$ is a hypersurface in $\mathbf{\Sigma}$, i.e., axially symmetric and has a simply connected interior, then in dimension greater than two the connected components of $\Sigma\setminus\Sigma'$ which are simply connected are exactly the upper and lower components (i.e., the connected components which intersect the coordinate axis $x_{n + 1}$). Hence, as a corollary we obtain that if the upper connected component $\mathcal{U}$ of $\Sigma\setminus\Sigma'$ is also regular then its image under $\Phi_{\Sigma}$ will be space like. Intuitively, it is clear that the other connected components cannot be mapped by $\Phi_{\Sigma}$ to a space like surface since their boundary is not connected while one can think of a space like surface of having a connected boundary at infinity.\\

Let us explain intuitively why the regularity condition is needed in Theorem 3.6 and for simplicity let us restrict our discussion only to the two dimensional plane.

First, we say that a path function $\gamma:I\rightarrow\Bbb R^{2}$ is well-ordered of degree $k\geq1$ if $\gamma$ is $k$ times continuously differentiable and $\gamma^{(i)}(t)\neq\bar{0}, t\in I, i = 1,...,k$. Secondly, observe that if the path function $\gamma$ is well-ordered of degree $k\geq1$ then in particular we have that $\gamma'(t)\neq\bar{0}, t\in I$, and thus we can all ways assume that $\gamma$ is parameterized in unit speed, i.e., the condition $|\gamma'(t)| = 1$ is satisfied for every $t\in I$.

Returning to our discussion of the regularity condition in Theorem 3.6, observe that since the definition of the map $\Phi_{\Sigma}$ involves the concept of the tangent plane, which itself involves the concept of a derivative, it follows that if $\gamma$ is well ordered of degree $k$ then $\Phi_{\Sigma}\circ\gamma$ will be well ordered of degree $k - 1$. That is, the map $\Phi_{\Sigma}$ reduces the degree of well-orderedness of $\gamma$ by one. Hence, if $\Sigma^{\ast}$ is a smooth subcurve of $\Sigma$, which is parameterized by $\gamma$, and we only know that $\gamma$ is well-ordered of degree one then the derivative of $\Phi_{\Sigma}\circ\gamma$ can vanish at some point which will imply in particular that $\Phi_{\Sigma}(\Sigma^{\ast})$ might not have a tangent line at that point and hence it will not be space like. However, this is where the regularity condition comes in order to guarantee the smoothness of $\Phi_{\Sigma}(\Sigma^{\ast})$. Indeed, if $\gamma$ is parameterized in unit speed then the condition $\gamma_{1}'(t)\gamma_{2}''(t) - \gamma_{1}''(t)\gamma_{2}'(t)\neq0$ is equivalent to $\gamma''(t)\neq\bar{0}$, i.e., $\gamma$ is well ordered of degree $2$ which will imply that $\Phi_{\Sigma}\circ\gamma$ is a well-ordered path function of degree one and thus $\Phi_{\Sigma}(\Sigma^{\ast})$ is smooth (i.e., has a tangent line at each point). In other words, when restricted to regular curves the map $\Phi_{\Sigma}$ is a diffeomorphism, i.e., the image of a regular curve under $\Phi_{\Sigma}$ contains no "spikes".

On a geometrical level the regularity condition on the subcurve $\Sigma^{\ast}$ implies that it is locally convex at every point. Indeed, in the two dimensional case local convexity of a curve $\Sigma^{\ast}$ at a point $p$ is equivalent to the condition that there exists a neighborhood of $p$ such that no two points in $\Sigma^{\ast}$ in this neighborhood have parallel tangent lines (of course these conditions are not equivalent in higher dimensions). The regularity condition guarantees that this last condition is satisfied. Indeed, let $\gamma$ be a parametrization of $\Sigma^{\ast}$ satisfying the condition $\gamma_{1}'(t)\gamma_{2}''(t) - \gamma_{1}''(t)\gamma_{2}'(t)\neq0$ and for the point $p$ assume that $\gamma(t_{0}) = p$. Since at least $\gamma_{1}'(t_{0})$ or $\gamma_{2}'(t_{0})$ are different from zero let us assume with out loss of generality that $\gamma_{1}'(t_{0})\neq 0$ and the last inequality is thus true in a neighborhood of $t_{0}$. Hence, for the function $s(t) = \gamma_{2}'(t) / \gamma_{1}'(t)$ the regularity condition implies that its derivative is different from zero in a neighborhood of $t_{0}$ which implies that $s(t)$ is injective in a neighborhood of $t_{0}$. But since $s(t)$ is the slope of the tangent lines of points in a neighborhood of $p$ this equivalently means that no two tangent lines in this neighborhood can have the same slope, i.e., they are not parallel.

Finally, in our case the regularity condition implies local convexity for regular connected components of $\Sigma\setminus\Sigma'$ in higher dimensions since these components are obtained as a surface of revolution of a regular curve.\\

For the proof of Theorem 3.6 we need to find the explicit form of $\Phi_{\Sigma}$ in case where the hypersurface $\Sigma$, given as in Theorem 3.6, is the set of solutions of $F(x) = 0$. In this case, if $\mathrm{H}_{x}(\Sigma) = \{y\in\Bbb R^{n + 1}: y\cdot\psi = \rho\}$ is the tangent plane to $\Sigma$ at $x\in\Sigma$ then we have that
\begin{equation}\hskip-1.5cm \psi = \frac{\nabla F(x)}{|(\nabla F)(x)|}, \rho = \frac{x\cdot \nabla F(x)}{|(\nabla F)(x)|} \textrm{ or } \psi = -\frac{\nabla F(x)}{|(\nabla F)(x)|}, \rho = - \frac{x\cdot \nabla F(x)}{|(\nabla F)(x)|}\end{equation}
where the sign is chosen so that $\rho \geq0$. In either case it can be easily seen that this does not change the right hand side of equation (3.1) in the definition of $\Psi$. Also, $\mathrm{H}_{x}(\Sigma)$ passes through $e_{n + 1}$ if and only if $x\cdot\nabla F(x) - F_{x_{n + 1}}(x) = 0$. Hence, using formula (3.1) for $\Psi$ and equation (3.2) we obtain the following corollary.

\begin{cor}

If a closed, smooth hypersurface $\Sigma$ in $\Bbb R^{n + 1}$, which is contained inside $\Bbb S^{n}$, is given by the equation $\Sigma: F(x) = 0$, then the surface map $\Phi_{\Sigma}$ is given by
$$\hskip-1.75cm\Phi_{\Sigma}:\Sigma\setminus\Sigma'\rightarrow\Bbb R^{n + 1} \textrm{ where } \Sigma' = \{x\in\Sigma:x\cdot\nabla F(x) - F_{x_{n + 1}}(x) = 0\},$$
\vskip-0.2cm
$$\hskip-12.9cm\Phi_{\Sigma}(x)$$
\begin{equation}\hskip0.6cm = \left(\frac{F_{x_{1}}(x)}{x\cdot\nabla F(x) - F_{x_{n + 1}}(x)},...,\frac{F_{x_{n}}(x)}{x\cdot\nabla F(x) - F_{x_{n + 1}}(x)},\frac{\sqrt{|\nabla F(x)|^{2} - (x\cdot\nabla F(x))^{2}}}{\left|x\cdot\nabla F(x) - F_{x_{n + 1}}(x)\right|}\right).\end{equation}

\end{cor}


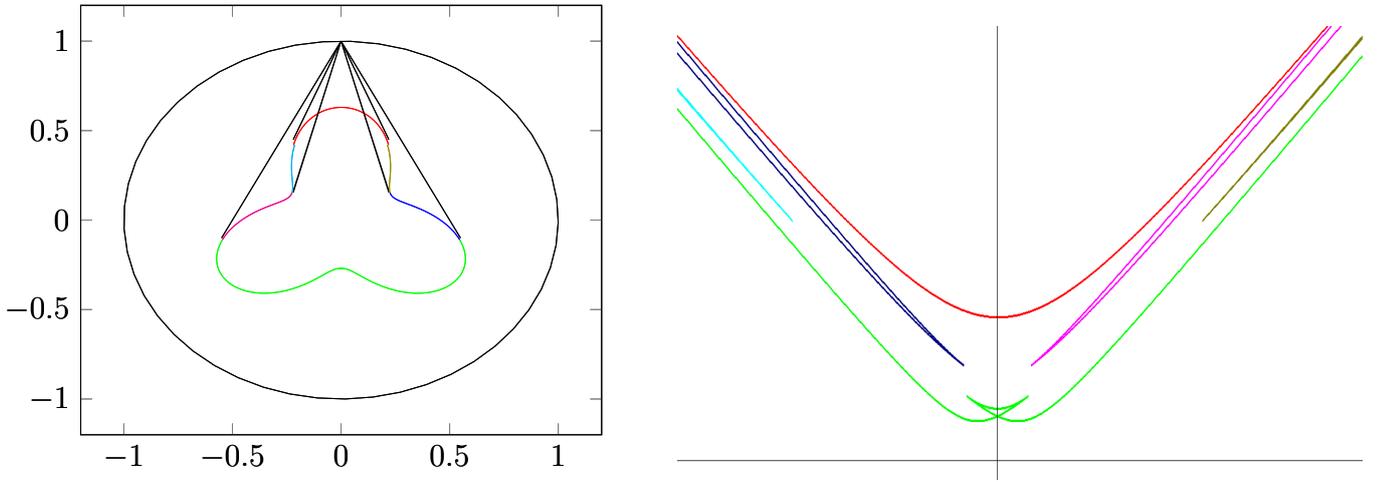
\begin{figure}[t]

\caption{\small{On the two dimensional plane let us consider the hypersurface $\Sigma$ which is parameterized by the path function $\gamma:[-\pi, \pi)\rightarrow\Bbb R^{2}, \gamma(\theta) = r(\theta)e^{i\theta}$, where $r(\theta) = 0.9 \cdot (0.5 + 0.2\cdot\sin(3\theta + 3.1))$. The set $\Sigma\setminus\Sigma'$, where $\Sigma'$ is the singularity set of $\Sigma$, consists of six connected components (see image on the left) and thus $\Phi_{\Sigma}(\Sigma\setminus\Sigma')$ consists also of six components (see the image on the right). Observe that in the connected components of $\Phi_{\Sigma}(\Sigma\setminus\Sigma')$ only the red part is a space like curve. This is because it is the only connected component whose inverse image under $\Phi_{\Sigma}$ is a regular subcurve. The other connected components are obtained as the image of $\Phi_{\Sigma}\circ\gamma$ on a domain for which there exists a point $\theta$ such that $\gamma_{1}'(\theta)\gamma_{2}''(\theta) - \gamma_{1}''(\theta)\gamma_{2}'(\theta) = 0$.}}

\vskip0.5cm

\hskip-1cm\begin{tikzpicture}

\begin{axis}

\addplot [blue, samples = 50, domain = -0.2 : 0.6]
(
    {0.9 * (0.5 + 0.2 * sin(deg(3 * x + 3.1))) * cos(deg(x))},
    {0.9 * (0.5 + 0.2 * sin(deg(3 * x + 3.1))) * sin(deg(x))}
);

\addplot [olive, samples = 50, domain = 0.6 : 1.1]
(
    {0.9 * (0.5 + 0.2 * sin(deg(3 * x + 3.1))) * cos(deg(x))},
    {0.9 * (0.5 + 0.2 * sin(deg(3 * x + 3.1))) * sin(deg(x))}
);

\addplot [magenta, samples = 50, domain = -0.2 : 0.6]
(
    {- 0.9 * (0.5 + 0.2 * sin(deg(3 * x + 3.1))) * cos(deg(x))},
    {0.9 * (0.5 + 0.2 * sin(deg(3 * x + 3.1))) * sin(deg(x))}
);

\addplot [cyan, samples = 50, domain = 0.6 : 1.1]
(
    {- 0.9 * (0.5 + 0.2 * sin(deg(3 * x + 3.1))) * cos(deg(x))},
    {0.9 * (0.5 + 0.2 * sin(deg(3 * x + 3.1))) * sin(deg(x))}
);

\addplot [green, samples = 50, domain = -1.6 : -0.2]
(
    {0.9 * (0.5 + 0.2 * sin(deg(3 * x + 3.1))) * cos(deg(x))},
    {0.9 * (0.5 + 0.2 * sin(deg(3 * x + 3.1))) * sin(deg(x))}
);

\addplot [green, samples = 50, domain = -1.6 : -0.2]
(
    {- 0.9 * (0.5 + 0.2 * sin(deg(3 * x + 3.1))) * cos(deg(x))},
    {0.9 * (0.5 + 0.2 * sin(deg(3 * x + 3.1))) * sin(deg(x))}
);

\addplot [red, samples = 50, domain = 1.1 : 2.06]
(
    {- 0.9 * (0.5 + 0.2 * sin(deg(3 * x + 3.1))) * cos(deg(x)) - 0.005},
    {0.9 * (0.5 + 0.2 * sin(deg(3 * x + 3.1))) * sin(deg(x))}
);

\addplot [black, samples = 50, domain = -pi / 4 - 0.125: 7 * pi / 4 - 0.125]
(
    {cos(deg(x))}, {sin(deg(x))}
);

\addplot [black, samples = 50, domain = -0.22 : 0]({x},{1 + 2.5 * x});

\addplot [black, samples = 50, domain = -0.22 : 0]({-x},{1 + 2.5 * x});

\addplot [black, samples = 50, domain = -0.55 : 0]({x},{1 + 2 * x});

\addplot [black, samples = 50, domain = -0.55 : 0]({-x},{1 + 2 * x});

\addplot [black, samples = 50, domain = -0.22 : 0]({x},{1 + 3.8 * x});

\addplot [black, samples = 50, domain = -0.22 : 0]({-x},{1 + 3.8 * x});

\end{axis}

\begin{axis}

\addplot [blue, samples = 50, domain = -0.2 : 0.6]
(
    {0.9 * (0.5 + 0.2 * sin(deg(3 * x + 3.1))) * cos(deg(x))},
    {0.9 * (0.5 + 0.2 * sin(deg(3 * x + 3.1))) * sin(deg(x))}
);

\addplot [olive, samples = 50, domain = 0.6 : 1.1]
(
    {0.9 * (0.5 + 0.2 * sin(deg(3 * x + 3.1))) * cos(deg(x))},
    {0.9 * (0.5 + 0.2 * sin(deg(3 * x + 3.1))) * sin(deg(x))}
);

\addplot [magenta, samples = 50, domain = -0.2 : 0.6]
(
    {- 0.9 * (0.5 + 0.2 * sin(deg(3 * x + 3.1))) * cos(deg(x))},
    {0.9 * (0.5 + 0.2 * sin(deg(3 * x + 3.1))) * sin(deg(x))}
);

\addplot [cyan, samples = 50, domain = 0.6 : 1.1]
(
    {- 0.9 * (0.5 + 0.2 * sin(deg(3 * x + 3.1))) * cos(deg(x))},
    {0.9 * (0.5 + 0.2 * sin(deg(3 * x + 3.1))) * sin(deg(x))}
);

\addplot [green, samples = 50, domain = -1.6 : -0.2]
(
    {0.9 * (0.5 + 0.2 * sin(deg(3 * x + 3.1))) * cos(deg(x))},
    {0.9 * (0.5 + 0.2 * sin(deg(3 * x + 3.1))) * sin(deg(x))}
);

\addplot [green, samples = 50, domain = -1.6 : -0.2]
(
    {- 0.9 * (0.5 + 0.2 * sin(deg(3 * x + 3.1))) * cos(deg(x))},
    {0.9 * (0.5 + 0.2 * sin(deg(3 * x + 3.1))) * sin(deg(x))}
);

\addplot [red, samples = 50, domain = 1.1 : 2.06]
(
    {- 0.9 * (0.5 + 0.2 * sin(deg(3 * x + 3.1))) * cos(deg(x)) - 0.005},
    {0.9 * (0.5 + 0.2 * sin(deg(3 * x + 3.1))) * sin(deg(x))}
);

\addplot [black, samples = 50, domain = -pi / 4 - 0.125: 7 * pi / 4 - 0.125]
(
    {cos(deg(x))}, {sin(deg(x))}
);

\addplot [black, samples = 50, domain = -0.22 : 0]({x},{1 + 2.5 * x});

\addplot [black, samples = 50, domain = -0.22 : 0]({-x},{1 + 2.5 * x});

\addplot [black, samples = 50, domain = -0.55 : 0]({x},{1 + 2 * x});

\addplot [black, samples = 50, domain = -0.55 : 0]({-x},{1 + 2 * x});

\addplot [black, samples = 50, domain = -0.22 : 0]({x},{1 + 3.85 * x});

\addplot [black, samples = 50, domain = -0.22 : 0]({-x},{1 + 3.85 * x});

\end{axis}

\end{tikzpicture}\begin{tikzpicture}

\hskip1cm\includegraphics[width = 9cm, height = 6cm]{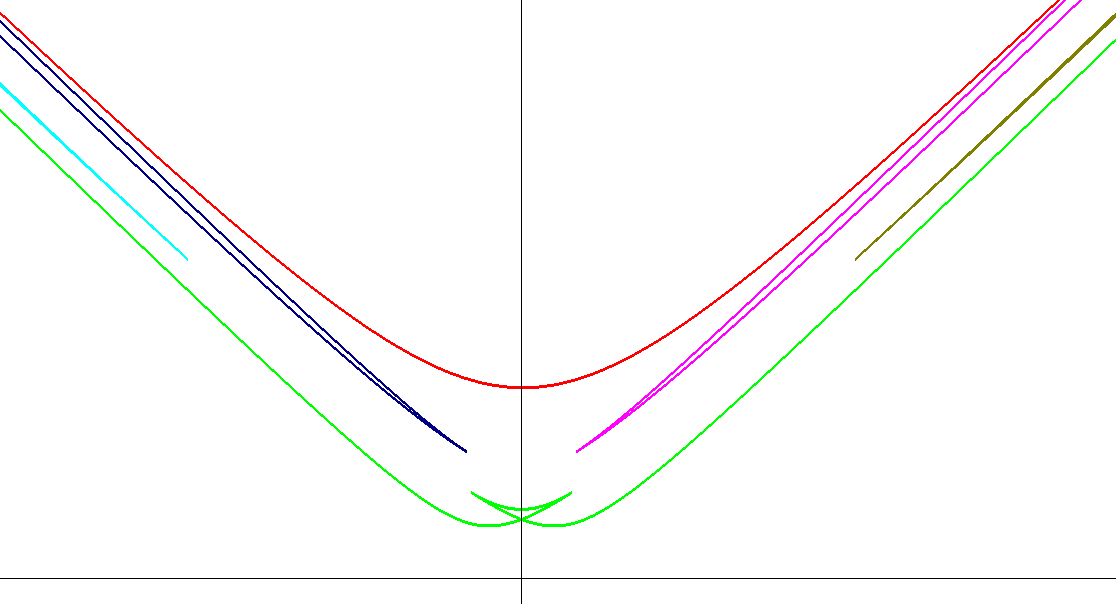}

\end{tikzpicture}

\end{figure}


\begin{example}

Let $\Sigma$ be the hypersphere
\vskip-0.2cm
$$\hskip-8.35cm\Sigma: F(x) = |x - \lambda\omega|^{2} - r^{2} = 0$$
in $\Bbb R^{n + 1}$ with radius $r$ and center at $\lambda\omega$ where $\omega\in\Bbb S^{n}$ and where $0\leq r < 1 - \lambda$ (the last condition guarantees that $\Sigma$ is contained inside $\Bbb S^{n}$). From Corollary 3.7 we have
\newpage
$$\hskip0.75cm\Phi_{\Sigma}:\{x\in\Bbb R^{n + 1}:|x - \lambda\omega|^{2} - r^{2} = 0, r^{2} - \lambda^{2} + \lambda(\omega\cdot x) - x_{n + 1} + \lambda\omega_{n + 1}\neq0\}\rightarrow\Bbb R^{n + 1}$$
$$\hskip-12.65cm\Phi_{\Sigma}(x)$$
$$\hskip0.3cm = \left(\frac{x_{1} - \lambda\omega_{1}}{r^{2} - \lambda^{2} + \lambda(\omega\cdot x) - x_{n + 1} + \lambda\omega_{n + 1}},...,\frac{x_{n} - \lambda\omega_{n}}{r^{2} - \lambda^{2} + \lambda(\omega\cdot x) - x_{n + 1} + \lambda\omega_{n + 1}},\right.$$
$$\hskip-5.3cm\left.\frac{\sqrt{r^{2} - (r^{2} - \lambda^{2} + \lambda(x\cdot\omega))^{2}}}{|r^{2} - \lambda^{2} + \lambda(\omega\cdot x) - x_{n + 1} + \lambda\omega_{n + 1}|}\right).$$
If $\omega^{\ast} = (\omega_{1},...,\omega_{n})$ denotes the orthogonal projection of $\omega$ into $\Bbb R^{n}$ then by standard considerations one can prove that the image $\Phi_{\Sigma}(\Sigma\setminus\Sigma')$ of $\Sigma\setminus\Sigma'$ under $\Phi_{\Sigma}$ has the algebraic representation as the union of the following two "upper" elliptic hyperboloids
\vskip-0.2cm
$$ A(y_{n + 1} \pm B)^{2} - C\left(\frac{\omega^{\ast}}{\sqrt{1 - \omega_{n + 1}^{2}}}\cdot \bar{y}\right)^{2} - D\left[|\bar{y}|^{2} - \left(\frac{\omega^{\ast}}{\sqrt{1 - \omega_{n + 1}^{2}}}\cdot\bar{y}\right)^{2}\right] = 1, y_{n + 1} > 0$$
where $\bar{y} = (y_{1},...,y_{n})$ and if we denote
$$\hskip-1.5cm Q = \sqrt{\frac{1 - r + \lambda\omega_{n + 1}}{1 + r - \lambda\omega_{n + 1}}}, P = \sqrt{\frac{1 + r + \lambda\omega_{n + 1}}{1 - r - \lambda\omega_{n + 1}}}, L = r + \lambda\sqrt{1 - \omega_{n + 1}^{2}}$$
then $A = 4 / (Q + P)^{2}, B = (Q - P) / 2$ and
$$\hskip-2.35cm C = \left(\sqrt{1 - L^{2}} - BL\right)^{2}A - L^{2}, D = \left(\sqrt{1 - r^{2}} - Br\right)^{2}A - r^{2}.$$
In particular, in case where $\lambda = 0$ (i.e., $\Sigma$ is a hypersphere with center at the origin and radius $r$) we have that
$$\hskip-1.5cm \Phi_{\Sigma}(\Sigma\setminus\Sigma'): \left(y_{n + 1} \pm \frac{r}{\sqrt{1 - r^{2}}}\right)^{2} - y_{1}^{2} - ... - y_{n}^{2} = \frac{1}{1 - r^{2}}, y_{n + 1} > 0.$$

Observe that when $\lambda = 0$ then $\Sigma$ is axially symmetric and, as according to Theorem 3.6, the images under $\Phi_{\Sigma}$ of the two connected components of $\Sigma\setminus\Sigma'$, which are regular and simply connected, are indeed space like.

\end{example}

\section{Proofs of the Main Results}

We begin with the proof of Theorem 3.6 since, as discussed above, Theorem 3.1 is based on its main result. The proof of Theorem 3.6 is based on the following lemma.

\begin{lem}

Let $\Sigma$ be a closed, smooth and connected curve in $\Bbb R^{2}$ which is contained inside the unit circle and let $\Sigma'$ be its set of singularity. Then, if $\mathcal{V}$ is a connected component of $\Sigma\setminus\Sigma'$ which is also regular then its image $\mathcal{V}' = \Phi_{\Sigma}(\mathcal{V})$ under $\Phi_{\Sigma}$ is a space like curve.

\end{lem}

\begin{proof}

First observe that since $\mathcal{V}$ is connected and $\Phi_{\Sigma}$ is continuous on $\Sigma\setminus\Sigma'$ (and in particular on $\mathcal{V}$) it follows that $\mathcal{V}' = \Phi_{\Sigma}(\mathcal{V})$ is also connected. Secondly, since the boundary points of $\mathcal{V}$ are mapped to infinity by $\Phi_{\Sigma}$ it follows that $\mathcal{V}'$ contains no boundary points. Hence, all is left is to prove that $\mathcal{V}'$ is smooth, i.e., has a tangent line at each point, and that it is space like, i.e., the normal to each tangent line is contained inside the set $x_{2}^{2} - x_{1}^{2}\geq0$.\\

Let us denote by $F = 0$ the equation which defines the curve $\Sigma$. Since $\mathcal{V}$ is regular it follows that it can be parameterized by a twice continuously differentiable path function $\gamma:I\rightarrow\Bbb R^{2}$ satisfying $\gamma_{1}'(t)\gamma_{2}''(t) - \gamma_{1}''(t)\gamma_{2}'(t)\neq0, t\in I$. Since $\mathcal{V}$ is contained in $\Sigma$ it follows that $F(\gamma_{1}(t), \gamma_{2}(t)) = 0, t\in I$ and hence by taking the derivative with respect to $t$ we obtain
\vskip-0.2cm
\begin{equation}\hskip-4.5cm\partial_{1}F(\gamma_{1}(t), \gamma_{2}(t))\gamma_{1}'(t) + \partial_{2}F(\gamma_{1}(t), \gamma_{2}(t))\gamma_{2}'(t) = 0, t\in I.\end{equation}

Now, let $p$ be a point in $\mathcal{V}'$ and our aim is to show that $\mathcal{V}'$ is smooth and space like in $p$ (i.e., has a tangent line in $p$ whose normal is contained in $x_{2}^{2} - x_{1}^{2}\geq0$). Let $t_{0}$ be a point in $I$ such that $\Phi_{\Sigma}(\gamma(t_{0})) = p$. Since $\Sigma$ is smooth it follows that either $\partial_{1}F(\gamma_{1}(t_{0}), \gamma_{2}(t_{0}))$ or $\partial_{2}F(\gamma_{1}(t_{0}), \gamma_{2}(t_{0}))$ are different from zero and let us assume, without loss of generality, that $\partial_{1}F(\gamma_{1}(t_{0}), \gamma_{2}(t_{0}))\neq 0$. By taking into account equation (4.1) this implies that $\gamma_{2}'(t_{0})\neq0$ and thus in a neighborhood of $t_{0}$ we have
\vskip-0.2cm
\begin{equation}\label{eq:99}\hskip-5.15cm\Rightarrow \partial_{2}F(\gamma_{1}(t), \gamma_{2}(t)) / \partial_{1}F(\gamma_{1}(t), \gamma_{2}(t)) = - \gamma_{1}'(t) / \gamma_{2}'(t).\end{equation}
Returning to the curve $\mathcal{V}$, observe that since $\Sigma$ is defined by the equation $F(x_{1}, x_{2}) = 0$ we have, from the definition of the map $\Phi_{\Sigma}$ in Corollary 3.7, the following parametrization
\begin{equation}\label{eq:100}\mathcal{V}' = \left\{\left(\frac{F_{x_{1}}(x)}{x\cdot\nabla F(x) - F_{x_{2}}(x)}, \frac{\sqrt{|\nabla F(x)|^{2} - (x\cdot \nabla F(x))^{2}}}{|x\cdot\nabla F(x) - F_{x_{2}}(x)|}\right): x = (x_{1}, x_{2})\in\mathcal{V}\right\}\end{equation}
of $\mathcal{V}'$. From the parametrization of $\mathcal{V}$ we can replace, in the right hand side of equation (\ref{eq:100}), the variable $x$ with $\gamma(t) = (\gamma_{1}(t), \gamma_{2}(t))$. Thus, using equation (\ref{eq:99}) it follows that in a neighborhood of $t_{0}$ we have the following parametrization of $\mathcal{V}'$ in a neighborhood of the point $p$:
\vskip-0.2cm
$$\hskip-2cm\mathcal{V}' : \left(\frac{F_{x_{1}}(\gamma(t))}{\gamma(t)\cdot\nabla F(\gamma(t)) - F_{x_{2}}(\gamma(t))}, \frac{\sqrt{|\nabla F(\gamma(t))|^{2} - (\gamma(t)\cdot \nabla F(\gamma(t)))^{2}}}{|\gamma(t)\cdot\nabla F(\gamma(t)) - F_{x_{2}}(\gamma(t))|}\right)$$
\begin{equation} = \left(\frac{\gamma_{2}'(t)}{\gamma_{1}(t)\gamma_{2}'(t) - \gamma_{2}(t)\gamma_{1}'(t) + \gamma_{1}'(t)}, \frac{\sqrt{[\gamma_{1}'(t)]^{2} + [\gamma_{2}'(t)]^{2} - (\gamma_{1}(t)\gamma_{2}'(t) - \gamma_{2}(t)\gamma_{1}'(t))^{2}}}{|\gamma_{1}(t)\gamma_{2}'(t) - \gamma_{2}(t)\gamma_{1}'(t) + \gamma_{1}'(t)|}\right).\end{equation}
\vskip0.2cm
\hskip-0.625cm By taking the derivative, with respect to $t$, of this parametrization it follows that the tangent line of $\mathcal{V}'$ at the point $p = \Phi_{\Sigma}(\gamma(t_{0}))$ has the direction
\vskip-0.2cm
$$\hskip-14.2cm \nu(t_{0})$$
$$\hskip0.1cm = K(t_{0})\left(1 - \gamma_{2}(t_{0}), \pm\frac{\gamma_{1}(t_{0})(\gamma_{1}(t_{0})\gamma_{2}'(t_{0}) - \gamma_{1}'(t_{0})\gamma_{2}(t_{0})) + \gamma_{1}(t_{0})\gamma_{1}'(t_{0}) + \gamma_{2}(t_{0})\gamma_{2}'(t_{0}) - \gamma_{2}'(t_{0})}{\sqrt{[\gamma_{1}'(t_{0})]^{2} + [\gamma_{2}'(t_{0})]^{2} - (\gamma_{1}(t_{0})\gamma_{2}'(t_{0}) - \gamma_{2}(t_{0})\gamma_{1}'(t_{0}))^{2}}}\right),$$
where
\vskip-0.2cm
$$\hskip-6.5cm K(t) = \frac{\gamma_{1}'(t)\gamma_{2}''(t) - \gamma_{1}''(t)\gamma_{2}'(t)}{(\gamma_{1}(t)\gamma_{2}'(t) - \gamma_{2}(t)\gamma_{1}'(t) + \gamma_{1}'(t))^{2}}.$$
Now, this is the crucial part where we use the fact that $\mathcal{V}$ is regular which translates to the condition that $\gamma_{1}'(t)\gamma_{2}''(t) - \gamma_{1}''(t)\gamma_{2}'(t)\neq0$ and thus $K(t)\neq0$.

Observe also that $K(t)$ is finite, i.e, $\gamma_{1}(t)\gamma_{2}'(t) - \gamma_{2}(t)\gamma_{1}'(t) + \gamma_{1}'(t) \neq 0$ since the last inequality is exactly the condition that the tangent lines of the image of the path function $(-\gamma_{1}(t), \gamma_{2}(t))$ (i.e., the reflection of $\mathcal{V}$ with respect to the $x_{2}$ axis) do not pass through the north pole $e_{2}$. This can be easily seen to be equivalent to the condition that the tangent lines of $\mathcal{V}$, which is parameterized by $(\gamma_{1}(t), \gamma_{2}(t))$, do not pass through the north pole $e_{2}$ which is the case since no point in $\Sigma'$ is in $\mathcal{V}$.

Hence, since $1 - \gamma_{2}(t_{0}) > 0$ (since $\gamma(t)$ is contained inside the unit circle) it follows that $\nu(t_{0})$ is finite and satisfies $\nu(t_{0})\neq\bar{0}$ and thus $\mathcal{V}'$ has a tangent line at $p$. To show that $\mathcal{V}'$ is space like in $p$ let us define
\vskip-0.2cm
$$\hskip-2.1cm h_{1}(t) = (1 - \gamma_{2}(t))\sqrt{[\gamma_{1}'(t)]^{2} + [\gamma_{2}'(t)]^{2} - (\gamma_{1}(t)\gamma_{2}'(t) - \gamma_{2}(t)\gamma_{1}'(t))^{2}},$$
$$\hskip-1.1cm h_{2}(t) = \pm[\gamma_{1}(t)(\gamma_{1}(t)\gamma_{2}'(t) - \gamma_{1}'(t)\gamma_{2}(t)) + \gamma_{1}(t)\gamma_{1}'(t) + \gamma_{2}(t)\gamma_{2}'(t) - \gamma_{2}'(t)]$$
then it can be easily verified that
\vskip-0.2cm
$$\hskip-1.2cm h_{1}^{2}(t) - h_{2}^{2}(t) = (1 - \gamma_{1}^{2}(t) - \gamma_{2}^{2}(t))(\gamma_{1}(t)\gamma_{2}'(t) - \gamma_{1}'(t)\gamma_{2}(t) + \gamma_{1}'(t))^{2}\geq0.$$
Hence, since $p$ is arbitrary it follows that $\mathcal{V}'$ is a space like curve.

\end{proof}

Observe that, in the two dimensional case, Lemma 4.1 is slightly stronger than Theorem 3.6 since the condition that $\Sigma$ must be axially symmetric is relaxed.

In higher dimensions Theorem 3.6 follows from Lemma 4.1 as follows. Let $\Sigma_{0}$ be the two dimensional curve in $\Bbb R^{2}$ obtained by intersecting $\Sigma$ with the two dimensional plane $\mathcal{H}$ spanned by $e_{1}$ and $e_{n + 1}$ and let $\Sigma_{0}'$ be its set of singularity. If $\mathcal{V}$ is simply connected and regular then it follows that the intersection of $\mathcal{V}$ with $\mathcal{H}$ is a subcurve $\mathcal{V}_{0}$ of $\Sigma_{0}$ which must be symmetric with respect to the second coordinate axis (since $\mathcal{V}$ is axially symmetric), connected (since $\mathcal{V}$ is simply connected) and regular (since $\mathcal{V}$ is regular and is the surface of revolution of $\mathcal{V}_{0}$). That is, $\mathcal{V}_{0}$ is a connected component of $\Sigma_{0}\setminus\Sigma_{0}'$ which is regular and thus by Lemma 4.1 its image $\Phi_{\Sigma_{0}}(\mathcal{V}_{0})$ under $\Phi_{\Sigma_{0}}$ is a space like curve which is also symmetric with respect to the second coordinate axis (since $\mathcal{V}_{0}$ is symmetric). Hence, since the image $\Phi_{\Sigma}(\mathcal{V})$ of $\mathcal{V}$ under $\Phi_{\Sigma}$ is the surface of revolution of $\Phi_{\Sigma_{0}}(\mathcal{V}_{0})$ it follows that $\Phi_{\Sigma}(\mathcal{V})$ is also space like and is also smooth since $\Phi_{\Sigma_{0}}(\mathcal{V}_{0})$ is symmetric (i.e., when revolving $\Phi_{\Sigma_{0}}(\mathcal{V}_{0})$ with respect to the second coordinate axis the resulting surface does not intersect itself). This proves Theorem 3.6.\\

For the proof of Theorem 3.1 we will also need the following two auxiliary lemmas. The first lemma, proved in \cite[Chap VI]{3}, is a uniqueness theorem for functions which satisfy the Darboux's equation and vanish, with their first order derivatives, on a subdomain of the hyperplane $t = 0$. The method of its proof can be easily extended to any smooth, space like surface. We present and prove this modified extended result in Lemma 4.2 below for the reader convenience. Before proving Lemma 4.2 we make the convention that a point $(p, p_{n + 1})\in\Bbb R^{n}\times\Bbb R^{+}$ is "above" a surface $S\subset\Bbb R^{n}\times\Bbb R^{+}$ if the segment $[(p, 0), (p, p_{n + 1})]$ intersects $S$ and the segment $[(p, p_{n + 1}), (p, \infty)]$ does not.

\begin{lem}

For a continuous function $f$, defined on $\Bbb R^{n}$, let $\mathcal{R}f$ be its spherical mean transform and assume that $\mathcal{R}f$ vanishes, together with its first order derivatives, on a space like surface $S\subset\Bbb R^{n}\times\Bbb R^{+}$. Then, if $(p, p_{n + 1})\in\Bbb R^{n}\times\Bbb R^{+}$ is a point above $S$ then $\mathcal{R}f$ vanishes at this point.

\end{lem}

\begin{proof}

For abbreviation we denote by $\mathcal{R}$ the spherical mean transform $\mathcal{R}f$ of $f$. Let us define the normalized spherical mean of $f$ by $\mathcal{Q}(x, t) = t^{1 - n}\mathcal{R}(x, t)$. Observe that $\mathcal{R}$ vanishes, together with its first order derivatives, on $S$ if and only if the same is true for $\mathcal{Q}$ and that $\mathcal{R}$ vanishes at $(p, p_{n + 1})$ if and only if the same is true for $\mathcal{Q}$. Hence, our aim is to show that $\mathcal{Q}$ vanishes at $(p, p_{n + 1})$ given the assumption that it vanishes on $S$ with its first order derivatives.

The normalized spherical mean transform $\mathcal{Q}(x,t)$ is known to satisfy the Darboux's equation (see \cite{1, 3, 12})
\vskip-0.2cm
$$\hskip-4.5cm L\left[\mathcal{Q}\right](x, t) = \left(\partial_{t}^{2} + \frac{n - 1}{t}\partial_{t} - \Delta_{x}\right)\mathcal{Q}(x, t) = 0.$$
Hence, in particular we have the identity
\begin{equation}\hskip-1.1cm 0 = - 2\mathcal{Q}_{t}L[\mathcal{Q}] = 2\sum_{i = 1}^{n}(\mathcal{Q}_{t}\mathcal{Q}_{x_{i}})_{x_{i}} - \left[\sum_{i = 1}^{n}\mathcal{Q}_{x_{i}}^{2} + \mathcal{Q}_{t}^{2}\right]_{t} - \frac{2(n - 1)}{t}\mathcal{Q}_{t}^{2}.\end{equation}
Let $P$ be the right circular cone with the apex at $(p, p_{n + 1})$. Then, if $G$ is the bounded domain between $P$ and $S$, $P'$ is the domain on the boundary of $G$ which belongs to $P$ and $S'$ is the domain on the boundary of $G$ which belongs to $S$ then by integrating identity (4.5) on $G$ and using the divergence theorem we have
\vskip-0.2cm
$$\hskip-2cm 0 = \int_{G}\left[\frac{2(n - 1)}{t}\mathcal{Q}_{t}^{2} + \left[\sum_{i = 1}^{n}\mathcal{Q}_{x_{i}}^{2} + \mathcal{Q}_{t}^{2}\right]_{t} - 2\sum_{i = 1}^{n}(\mathcal{Q}_{t}\mathcal{Q}_{x_{i}})_{x_{i}}\right]dxdt$$
$$ = \int_{G}\left[\frac{2(n - 1)}{t}\mathcal{Q}_{t}^{2}\right]dxdt + \int_{P'}\left[\left[\sum_{i = 1}^{n}\mathcal{Q}_{x_{i}}^{2} + \mathcal{Q}_{t}^{2}\right]\hat{n}_{t} - 2\sum_{i = 1}^{n}(\mathcal{Q}_{t}\mathcal{Q}_{x_{i}})\hat{n}_{x_{i}}\right]dP'$$
\begin{equation}\hskip-4.5cm + \int_{S'}\left[\left[\sum_{i = 1}^{n}\mathcal{Q}_{x_{i}}^{2} + \mathcal{Q}_{t}^{2}\right]\hat{n}_{t} - 2\sum_{i = 1}^{n}(\mathcal{Q}_{t}\mathcal{Q}_{x_{i}})\hat{n}_{x_{i}}\right]dS'.\end{equation}
By our assumption, all the derivatives of $\mathcal{Q}$ vanish on $S$, and in particular on $S'$, and thus the third integral in the right hand side of equation (4.6) vanishes. On the second integral the normal $\hat{n} = (\hat{n}_{x_{1}},...,\hat{n}_{x_{n}}, \hat{n}_{t})$ to $P'$ satisfies $\hat{n}_{t}^{2} = \hat{n}_{x_{1}}^{2} + ... + \hat{n}_{x_{n}}^{2}$. Thus, on $P'$ we have the following identity
$$\hskip-1.5cm\left[\sum_{i = 1}^{n}\mathcal{Q}_{x_{i}}^{2} + \mathcal{Q}_{t}^{2}\right]\hat{n}_{t} - 2\sum_{i = 1}^{n}(\mathcal{Q}_{t}\mathcal{Q}_{x_{i}})\hat{n}_{x_{i}} = \frac{1}{\hat{n}_{t}}\sum_{i = 1}^{n}\left(\mathcal{Q}_{x_{i}}\hat{n}_{t} - \mathcal{Q}_{t}\hat{n}_{x_{i}}\right)^{2}.$$
Now, observe that since the surface $S$ is space like and the point $(p, p_{n + 1})$ is above $S$ then, for the domain $G$, the exterior normals at its boundary points which belong $P'$ are pointing upward. That is, $P'$ is the "upper" part of the boundary of $G$ and $S'$ is its "lower part".
Hence, $\hat{n}_{t}\geq0$ on $P'$ and thus the integrand in the integral on $P'$ in the right hand side of equation (4.6) is nonnegative. Observe also that the integrand on $G$ is also nonnegative since $\mathcal{Q}_{t}^{2}\geq0$ and $t\geq0$ (because $G\subset\Bbb R^{n}\times\Bbb R^{+})$. Thus, both integrands vanish on their domain of integration respectively, i.e.,
\vskip-0.2cm
\begin{equation}\hskip-5.9cm\frac{1}{\hat{n}_{t}}\sum_{i = 1}^{n}\left(\mathcal{Q}_{x_{i}}\hat{n}_{t} - \mathcal{Q}_{t}\hat{n}_{x_{i}}\right)^{2} = 0, (x,t)\in P',\end{equation}
\begin{equation}\hskip-8cm\frac{n - 1}{t}\mathcal{Q}_{t}^{2}(x, t) = 0, (x, t)\in G.\end{equation}
Since $P'\subset G$ it follows in particular, from equation (4.8), that if $n\geq 2$ then $\mathcal{Q}_{t}$ vanishes on $P'$ and thus from equation (4.7) we have $\mathcal{Q}_{x_{i}}(x, t) = 0, i = 1,..., n$ for every $(x, t)\in P'$ (the case where $n = 1$, where the Darboux's equation coincides with the wave equation, is even simpler, see \cite[Chap VI, p. 643]{3}). Hence, all the first order derivatives of $\mathcal{Q}$ vanish on $P'$ and thus $\mathcal{Q}$ is constant on $P'$. Indeed, if $\varphi:A\rightarrow P' (A\subset\Bbb R^{n})$ is a parametrization of $P'$ then for the function $U(x) = \mathcal{Q}(\varphi(x))$ we have, by the chain rule, that $U_{x_{1}} = ... = U_{x_{n}} = 0$. Thus, $U$ is constant on $A$ which is equivalent to the fact that $\mathcal{Q}$ is constant on $P'$. Finally, since $\mathcal{Q}$ is constant on $P'$ and vanishes on the intersection of $P'$ with the surface $S$ then it follows that $\mathcal{Q}$ vanishes on $P'$. Since $(p, p_{n + 1})\in P'$ it follows in particular that $\mathcal{Q}$ vanishes on $(p, p_{n + 1})$. This finishes the proof of Lemma 4.2.

\end{proof}

For the proof of the second lemma we introduce the following map
$$\hskip-4.1cm\Psi_{1}:\Bbb B(0,1)\setminus\Bbb S_{0}\rightarrow\Bbb R^{n}\times\Bbb R^{+}, \Psi_{1}(x) = \Psi\left(\frac{x}{|x|}, |x|\right).$$
The map $\Psi_{1}$ acts as follows: for a point $x$ in $\Bbb B(0,1)\setminus\Bbb S_{0}$ we consider the unique hyperplane whose closest point to the origin is $x$ (since $x\neq0$ this hyperplane is well defined). This hyperplane intersects $\Bbb S^{n}$ at a subsphere which is then projected into a subsphere in $\Bbb R^{n}$ by using the stereographic projection. Finally, we identify this projected sphere with the point in $\Bbb R^{n}\times\Bbb R^{+}$ which corresponds to its center and radius respectively.

Observe that for points $x = \rho\psi$ in the punctured sphere $\Bbb S_{0}\setminus\{\bar{0}\}$ we have that $\rho = \psi_{n + 1}$. That is, these are the closest distance points, to the origin, of hyperplanes that pass through the north pole $e_{n + 1}$. Hence, the values $\Psi(\psi, \rho)$ are not defined for these points and hence we have to omit them from the domain of definition of $\Psi_{1}$. The point $x = \bar{0}$ is also omitted since it does not determine a hyperplane in $\Bbb R^{n + 1}$ in a unique way. It can be easily proved that $\Psi_{1}$ maps the interior of $\Bbb S_{0}$ into the domain in $\Bbb R^{n}\times\Bbb R^{+}$ which is above the upper hyperboloid $\mathcal{H}:x_{n + 1}^{2} - x_{1}^{2} - ... - x_{n}^{2} = 1, x_{n + 1} > 0$, and the exterior of $\Bbb S_{0}$, in the unit ball $\Bbb B(0,1)$, into the domain in $\Bbb R^{n}\times\Bbb R^{+}$ which is below $\mathcal{H}$ (see Figure 5). Explicitly, $\Psi_{1}$ is given by
$$\hskip-4cm\Psi_{1}(x) = \left(\frac{x_{1}}{|x|^{2} - x_{n + 1}},...,\frac{x_{n}}{|x|^{2} - x_{n + 1}}, \frac{|x|\sqrt{1 - |x|^{2}}}{\left||x|^{2} - x_{n + 1}\right|}\right).$$

\begin{figure}

\caption{\small{The left image below consists of the two connected parts of the domain of definition of $\Psi_{1}$. The blue part is the interior of $\Bbb S_{0}$ and the light blue part is its complement in the unit disk. The right image below consists of the images of the interior of $\Bbb S_{0}$ and its complement respectively under $\Psi_{1}$ which are the domains above (blue part) and below (light blue part) the upper part of the hyperboloid $x_{n + 1}^{2} - x_{1}^{2} - ... - x_{n}^{2} = 1$.}}

\hskip-1cm\begin{tikzpicture}

\draw [fill, cyan] (0, 0) circle (3cm);

\draw [fill, blue] (0, 1.5) circle (1.5cm);

\draw [very thick, white] (0, 0) circle (4cm);

\draw [very thick, white] (0, 1.5) circle (1.5cm);

\draw[thick,->] (-3.5,0) -- (3.5,0);
\draw[thick,->] (0,-3.5) -- (0,3.5);

\end{tikzpicture}\hskip1cm\begin{tikzpicture}

\fill [blue, domain=-3:3, variable=\x]
  (-3, 0)
  -- plot ({\x}, {3.2})
  -- (3, 0);

  \fill [cyan, domain=-3:3, variable=\x]
  (-3, 0)
  -- plot ({\x}, {sqrt(1 + \x*\x)})
  -- (3, 0);

  \draw[very thick, white, domain=-3:3, smooth, variable=\t] plot ({\t},{sqrt(1 + \t * \t)});

\draw[thick,->] (-3.5,0) -- (3.5,0);
\draw[thick,->] (0,-3.5) -- (0,3.5);

\end{tikzpicture}

\end{figure}
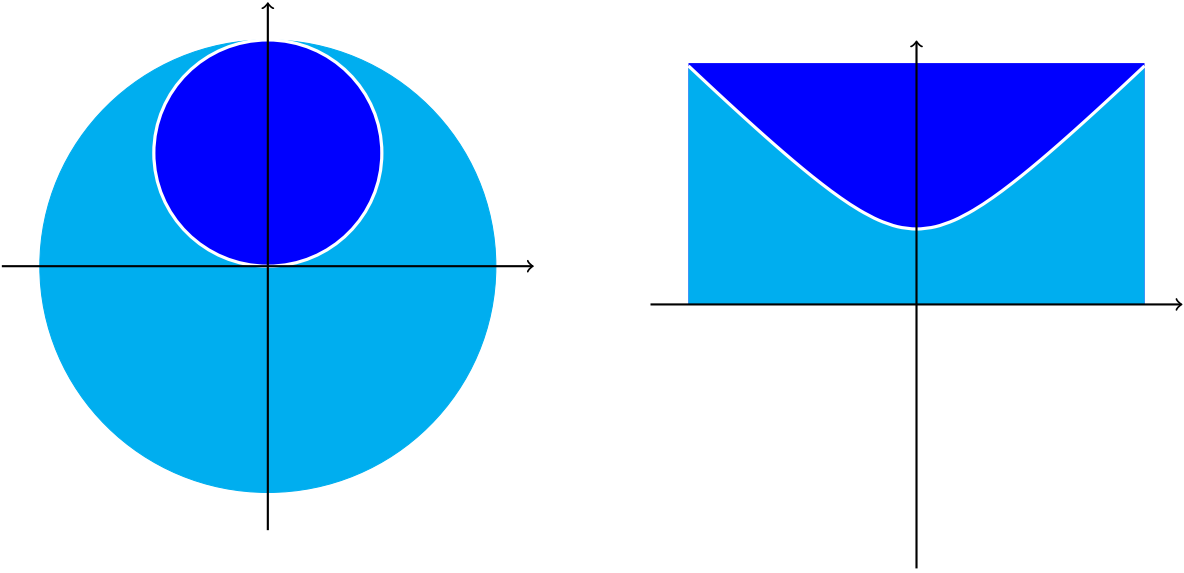

\begin{lem}

Let $\Sigma$ be in $\mathbf{\Sigma}$ and let $\mathcal{U}$ be the upper connected component of $\Sigma\setminus\Sigma'$ and assume that it is also regular. Let $f$ be a function in $C_{0}(\Bbb S^{n})$ and suppose that the spherical transform $\mathcal{S}f$ of $f$ vanishes on the family of subspheres $\Bbb S_{x}(\Sigma), x\in\mathcal{U}$, together with their infinitesimal perturbations. Then, if $\mathrm{H}$ is a tangent plane to $\mathcal{U}\setminus\Sigma_{0}$ then $f$ vanishes on the connected component of the set $\Bbb S^{n}\setminus\mathrm{H}$ which contains the north pole $e_{n + 1}$.

\end{lem}

\begin{proof}

The following relation $\Phi_{\Sigma}|_{\Sigma\setminus\Sigma_{0}} = \Psi_{1}\circ\Psi_{0}$, which follows immediately from the definitions of the maps $\Phi_{\Sigma}, \Psi_{0}$ and $\Psi_{1}$, will be used throughout the proof of the lemma. Using the relation (2.2) between the spherical and spherical mean transforms we have
$$(\mathcal{S}_{0}f)(x) = (\mathcal{S}f)(x / |x|, |x|) = (\mathcal{R}g)\left(\frac{x^{\ast}}{|x|^{2} - x_{n + 1}}, \frac{|x|\sqrt{1 - |x|^{2}}}{||x|^{2} - x_{n + 1}|}\right) = (\mathcal{R}g)(\Psi_{1}(x)),$$
where $x^{\ast} = (x_{1},...,x_{n})$ and the relation between $f$ and $g$ is given by equation (2.3). Hence, we obtain that
\begin{equation}\hskip-2cm(\mathcal{S}_{0}f)(x) = (\mathcal{R}g)(\Psi_{1}(x)), (\mathcal{S}_{0}f)_{x_{i}}(x) = [(\mathcal{R}g)(\Psi_{1}(x))]_{x_{i}}, i = 1, ..., n + 1.\end{equation}
The right hand side of equation (4.9) can be written equivalently as
\begin{equation}\hskip-6cm\nabla(\mathcal{S}_{0}(f))(x) = \mathbf{J}\Psi_{1}(x)^{\mathrm{T}}\cdot\nabla(\mathcal{R}g)(\Psi_{1}(x))\end{equation}
where in the right hand side we have a matrix multiplication of the Jacobian matrix $\mathbf{J}\Psi_{1}$ of $\Psi_{1}$, evaluated at the point $x$, and the gradient vector $\nabla(\mathcal{R}g)$ of $\mathcal{R}g$ evaluated at the point $\Psi_{1}(x)$.
Now, the Jacobian of $\Psi_{1}$ is given by
\vskip-0.2cm
$$\hskip-6.5cm |\det(\mathbf{J}\Psi_{1})(x)| = \frac{|x|}{||x|^{2} - x_{n + 1}|^{n + 1}\sqrt{1 - |x|^{2}}}$$
and so it does not vanish in the domain of definition of $\mathcal{S}_{0}f$. Hence, if all the derivatives of $\mathcal{S}_{0}f$ vanish at the point $x$ then from equation (4.10) it follows that all the derivatives of $\mathcal{R}g$ vanish at the point $\Psi_{1}(x)$. By assumption, the spherical transform $\mathcal{S}f$ vanishes on the family of subspheres $\Bbb S_{x}(\Sigma), x\in\mathcal{U}$, together with their infinitesimal perturbations. Hence, from Definition 2.8 we have
\vskip-0.2cm
\begin{equation}\hskip-2cm(\mathcal{S}_{0}f)(\Psi_{0}(x)) = 0, \partial_{i}(\mathcal{S}_{0}f)(\Psi_{0}(x)) = 0, i = 1,..., n + 1, \forall x\in\mathcal{U}\setminus\Sigma_{0}.\end{equation}
That is, all the derivatives of $\mathcal{S}_{0}f$ vanish at points $y = \Psi_{0}(x)$ where $x\in\mathcal{U}\setminus\Sigma_{0}$ and thus all the derivatives of $\mathcal{R}g$ vanish at points $z = \Psi_{1}(y) = \Psi_{1}(\Psi_{0}(x)) = \Phi_{\Sigma}(x)$ where $x\in\mathcal{U}\setminus\Sigma_{0}$. Also, by replacing $x$ with $\Psi_{0}(x)$ in equation (4.9) we have that
$$\hskip-5cm(\mathcal{S}_{0}f)(\Psi_{0}(x)) = (\mathcal{R}g)(\Psi_{1}(\Psi_{0}(x))) = (\mathcal{R}g)(\Phi_{\Sigma}(x)).$$
Thus, since, from equation (4.11), $(\mathcal{S}_{0}f)\circ\Psi_{0}$ vanishes on $\mathcal{U}\setminus\Sigma_{0}$ then the same is true for $(\mathcal{R}g)\circ\Phi_{\Sigma}$. That is, $\mathcal{R}g$ vanishes at each point $\Phi_{\Sigma}(x)$ where $x\in\mathcal{U}\setminus\Sigma_{0}$.\\

Hence, in conclusion we obtained that $\mathcal{R}g$ vanishes, with its first order derivatives, on the image of the set $\mathcal{U}\setminus\Sigma_{0}$ under the map $\Phi_{\Sigma}$. Observe that since the image of $\Sigma_{0}\cap\mathcal{U}$, under $\Phi_{\Sigma}$, is a set of codimension $1$ as a subset of the image $\mathcal{U}' := \Phi_{\Sigma}(\mathcal{U})$, and since $\mathcal{R}g$ is continuously differentiable, we can assume that $\mathcal{R}g$, and its first order derivatives, vanish on the whole image $\mathcal{U}'$.

Now, let $u\in\mathcal{U}$ be any point of tangency in the intersection of $\mathcal{U}$ with $\mathrm{H}$ and let $u' = (u'', u_{n + 1}')$ ($u''\in\Bbb R^{n}, u_{n + 1}'\in\Bbb R^{+}$) be its image in $\mathcal{U}'$ under the map $\Phi_{\Sigma}$. Since, by Theorem 3.6, $\mathcal{U}'$ is a space like surface (since $\mathcal{U}$ is regular and it is simply connected because it is the upper connected component of $\Sigma\setminus\Sigma'$) it follows that each point in the interior $P_{\mathrm{int}} = |t - u_{n + 1}'|^{2}\geq|x - u''|^{2}, t\geq u_{n + 1}'$, of the right circular cone with apex at $u'$ is above $\mathcal{U}'$. Hence, from Lemma 4.2 it follows that $\mathcal{R}g$ vanishes at each point in $P_{\mathrm{int}}$. Thus, if $\mathbf{S}$ denotes the hypersphere in $\Bbb R^{n}$ with the center at $u''$ and radius $u_{n + 1}'$ then the last condition is equivalent to the condition that the integrals of $g$ vanish on each hypersphere in $\Bbb R^{n}$ which contains $\mathbf{S}$. Since $f$ is continuous and vanishes in a neighborhood of $e_{n + 1}$ it follows that $g$ is continuous and is compactly supported and thus from \cite[Chap I, Lemma 2.7]{11} it follows that $g$ vanishes outside $\mathbf{S}$. From the definition of the stereographic projection $\Lambda$ and the map $\Phi_{\Sigma}$ it follows that the inverse image of $\mathbf{S}$ under $\Lambda$ is the subsphere $\mathbf{S}'$ on $\Bbb S^{n}$ obtained by intersecting $\Bbb S^{n}$ with the hyperplane $\mathrm{H}$. Thus, the inverse image of the exterior of $\mathbf{S}$ under $\Lambda$ must coincide with one of the connected components of $\Bbb S^{n}\setminus\mathbf{S}'$ and in fact it must coincide with the connected component $\mathbf{C}$ which contains the north pole $e_{n + 1}$ since it is the only component which is mapped to infinity (i.e., to an unbounded set) by $\Lambda$. Now, from the relation between the functions $f$ and $g$ it is clear that the support of $f$ coincides with the inverse image of the support of $g$ under $\Lambda$ and thus $f$ vanishes in $\mathbf{C}$. This proves Lemma 4.3.

\end{proof}

\textbf{Proof of Theorem 3.1:} For a hyperplane $\mathrm{H}$ in $\Bbb R^{n + 1}$, which intersects the unit sphere $\Bbb S^{n}$ and does not pass through the north pole $e_{n + 1}$, let us denote by $\mathrm{S}_{\mathrm{H}}^{+}$ and $\mathrm{S}_{\mathrm{H}}^{-}$ respectively the connected components of $\Bbb S^{n}\setminus\mathrm{H}$ where $\mathrm{S}_{\mathrm{H}}^{+}$ is the connected component which contains $e_{n + 1}$. Let us also denote by $\mathrm{H}^{+}$ and $\mathrm{H}^{-}$ respectively the halfspaces generated by $\mathrm{H}$ which correspond to $\mathrm{S}_{\mathrm{H}}^{+}$ and $\mathrm{S}_{\mathrm{H}}^{-}$. That is,
$$\hskip-8.5cm\mathrm{S}_{\mathrm{H}}^{+} = \Bbb S^{n}\cap\mathrm{H}^{+}, \mathrm{S}_{\mathrm{H}}^{-} = \Bbb S^{n}\cap\mathrm{H}^{-}.$$

Now let $\mathrm{H}_{0}$ be a tangent plane to $\mathcal{U}$ which passes through $e_{n + 1}$, i.e., $\mathrm{H}_{0}$ passes through the boundary of $\mathcal{U}$. Let us denote by $\mathrm{H}_{0}^{+}$ and $\mathrm{H}_{0}^{-}$ the half spaces generated by $\mathrm{H}_{0}$. This tangent plane is also tangent to the tangent cone $C_{\Sigma}$ of $\Sigma$ (since both sets contain the segment between $e_{n + 1}$ and a boundary point of $\mathcal{U}$). Hence, one of the half spaces $\mathrm{H}_{0}^{+}$ or $\mathrm{H}_{0}^{-}$ must contain $C_{\Sigma}$ and we can assume, without loss of generality, that this is the half space $\mathrm{H}_{0}^{-}$. Let us also denote $\mathrm{S}_{\mathrm{H}_{0}}^{+} = \Bbb S^{n}\cap\mathrm{H}_{0}^{+}$ and $\mathrm{S}_{\mathrm{H}_{0}}^{-} = \Bbb S^{n}\cap\mathrm{H}_{0}^{-}$.

Now, since $\mathrm{H}_{0}^{+}$ passes through the north pole $e_{n + 1}$ we cannot use Lemma 4.3 for this tangent plane. However, since $\mathrm{H}_{0}$ passes through the boundary of $\mathcal{U}$ it follows that there is a series of tangent planes $\{\mathrm{H}_{i}\}_{i}$ of $\mathcal{U}$ which converge to $\mathrm{H}_{0}$ and do not pass through $e_{n + 1}$. By taking a subsequence we can assume that none of these tangent planes pass through the origin (i.e, these are tangent planes to points on $\mathcal{U}\setminus\Sigma_{0}$). Indeed, otherwise it follows that $\mathrm{H}_{0}$ also passes through the origin and since it also passes through $e_{n + 1}$ and is tangent to $\mathcal{U}$ it follows, since $\mathcal{U}$ is axially symmetric, that $\mathcal{U}$ must intersect itself. However, by the assumption on smoothness on $\Sigma$ this is impossible.

Now, let us consider the subsequence of tangent planes $\{\mathrm{H}_{i}\}_{i}$ which intersect the segment $l$ between $-e_{n + 1}$ and $e_{n + 1}$ and the complement subsequence. Since at least one of these subsequences is infinite then by passing again to a subsequence we have the following two possibilities:\\

\hskip-0.6cm $\bullet$ Each tangent plane $\mathrm{H}_{i}$ intersects the segment $l$.\\

In this case, since $\{\mathrm{H}_{i}\}_{i}$ converge to $\mathrm{H}_{0}$ and $\mathrm{H}_{0}$ intersects the segment $l$ only at the point $e_{n + 1}$, it follows that each hyperplane $\mathrm{H}_{i}$ intersects $l$ at a point $a_{i}$ such that $a_{i}\rightarrow e_{n + 1}$. Hence, it
follows that each half space $\mathrm{H}_{i}^{+}$ of $\mathrm{H}_{i}$, which contains the north pole, contains only an infinitesimally small neighborhood of $e_{n + 1}$ in the segment $l$ (more specifically, this neighborhood is the segment $[a_{i}, e_{n + 1}]$). Since the interior of the cone $C_{\Sigma}$ contains the whole segment $l\setminus\{\pm e_{n + 1}\}$ and since $\{\mathrm{H}_{i}\}_{i}$ converge to a tangent plane of $C_{\Sigma}$ it follows that as $i\rightarrow\infty$ the intersection of $\mathrm{H}_{i}^{+}$ with the interior of $C_{\Sigma}$ converges to the empty set. Hence, the sequence of half spaces $\mathrm{H}_{i}^{+}$ "converge" to the half space $\mathrm{H}_{0}^{+}$ of $\mathrm{H}_{0}$ which does not contain $C_{\Sigma}$. Now, using Lemma 4.3 the function $f$ vanishes on each connected component $\mathrm{S}_{\mathrm{H}_{i}}^{+}$ of $\Bbb S^{n}\setminus\mathrm{H}_{i}$ obtained by intersecting $\Bbb S^{n}$ with the half space $\mathrm{H}_{i}^{+}$. Hence, since the family of half spaces $\{\mathrm{H}_{i}^{+}\}_{i}$ converge to the half space $\mathrm{H}_{0}^{+}$ and $\mathrm{S}_{\mathrm{H}_{i}}^{+} = \Bbb S^{n}\cap\mathrm{H}_{i}^{+}, \mathrm{S}_{\mathrm{H}_{0}}^{+} = \Bbb S^{n}\cap\mathrm{H}_{0}^{+}$ it follows that the family of sets $\{\mathrm{S}_{\mathrm{H}_{i}}^{+}\}_{i}$ converges to $\mathrm{S}_{\mathrm{H}_{0}}^{+}$ and thus $f$ vanishes on this set. Now, we can repeat the above argument for every rotation $\mathrm{H}_{0}'$ of $\mathrm{H}_{0}$, which preserves the last coordinate $x_{n + 1}$, and use the fact that since $\mathcal{U}$ is axially symmetric we will have the same behavior of the tangent planes which approach $\mathrm{H}_{0}'$ (i.e., each tangent plane intersects $l$). Hence, we can conclude that $f$ vanishes on each rotation of $\mathrm{S}_{\mathrm{H}_{0}}^{+}$, which preserves the last coordinate axis $x_{n + 1}$, and use the fact that the union of these sets, which is obtained by rotating $\mathrm{S}_{\mathrm{H}_{0}}^{+}$, covers the exterior of the cone $C_{\Sigma}$ in the sphere $\Bbb S^{n}$. That is, $f$ is supported on the projected set $\Pi_{\Sigma}$ of the cone $C_{\Sigma}$.

\begin{figure}

\caption{\small{The left image corresponds to the case where the series $\{\mathrm{H}_{i}\}_{i}$ converges to $\mathrm{H}_{0}$, each plane in this series is tangent to a point in $\mathcal{U}$ (the red part) and also intersects the segment $l$. The right image corresponds to the case where these planes do not intersect $l$.}}

\hskip-1cm\begin{tikzpicture}

\begin{axis}

\addplot [blue, samples = 50, domain = -3.2 / 2 : 1.1]
(
    {(0.5 + 0.2 * sin(deg(3 * x + 3.1))) * cos(deg(x))},
    {(0.5 + 0.2 * sin(deg(3 * x + 3.1))) * sin(deg(x)) - 0.2}
);

\addplot [blue, samples = 50, domain = -3.2 / 2 : 1.1]
(
    {- (0.5 + 0.2 * sin(deg(3 * x + 3.1))) * cos(deg(x))},
    {(0.5 + 0.2 * sin(deg(3 * x + 3.1))) * sin(deg(x)) - 0.2}
);

\addplot [red, samples = 50, domain = 1.1 : 2.06]
(
    {- (0.5 + 0.2 * sin(deg(3 * x + 3.1))) * cos(deg(x)) - 0.005},
    {(0.5 + 0.2 * sin(deg(3 * x + 3.1))) * sin(deg(x)) - 0.2}
);

\addplot [black, samples = 50, domain = -pi :  pi ]
(
    {cos(deg(x))}, {sin(deg(x))}
);

\addplot [black, samples = 50, domain = -0.7 : 0]({x + 0.04},{1 + 2.5 * x});

\addplot [black, samples = 50, domain = -1.29 : -0.123]({x + 0.4},{1.1 + 1.2 * x});

\addplot [black, samples = 50, domain = -1.95 : -0.25]({x + 1},{0.71 + 0.2 * x});

\addplot [black, dashed, samples = 50, domain = -1 : 1]({0},{x});

\draw [fill] (1000, 2000) circle (0.05cm);

\end{axis}

\begin{axis}

\addplot [blue, samples = 50, domain = -3.2 / 2 : 1.1]
(
    {(0.5 + 0.2 * sin(deg(3 * x + 3.1))) * cos(deg(x))},
    {(0.5 + 0.2 * sin(deg(3 * x + 3.1))) * sin(deg(x)) - 0.2}
);

\addplot [blue, samples = 50, domain = -3.2 / 2 : 1.1]
(
    {- (0.5 + 0.2 * sin(deg(3 * x + 3.1))) * cos(deg(x))},
    {(0.5 + 0.2 * sin(deg(3 * x + 3.1))) * sin(deg(x)) - 0.2}
);

\addplot [red, samples = 50, domain = 1.1 : 2.06]
(
    {- (0.5 + 0.2 * sin(deg(3 * x + 3.1))) * cos(deg(x)) - 0.005},
    {(0.5 + 0.2 * sin(deg(3 * x + 3.1))) * sin(deg(x)) - 0.2}
);

\addplot [black, samples = 50, domain = -pi :  pi ]
(
    {cos(deg(x))}, {sin(deg(x))}
);

\addplot [black, samples = 50, domain = -0.7 : 0]({x + 0.04},{1 + 2.5 * x});

\addplot [black, samples = 50, domain = -1.29 : -0.123]({x + 0.4},{1.1 + 1.2 * x});

\addplot [black, samples = 50, domain = -1.95 : -0.25]({x + 1},{0.71 + 0.2 * x});

\addplot [black, dashed, samples = 50, domain = -1 : 1]({0},{x});

\draw [fill] (1000, 2000) circle (0.05cm);

\end{axis}

\end{tikzpicture}\begin{tikzpicture}

\begin{axis}

\addplot [black, samples = 50, domain = -pi :  pi ]
(
    {cos(deg(x))}, {sin(deg(x))}
);

\addplot [red, samples = 50, domain =  - 5 * pi / 4 : pi / 4]
(
    {0.3 * cos(deg(x))}, {0.5 + 0.3 * sin(deg(x))}
);

\addplot [blue, samples = 50, domain =  -pi - pi / 3 - 0.145 : pi / 3 + 0.15]
(
    {0.6 * cos(deg(x))}, {0.15 + 0.6 * sin(deg(x))}
);

\addplot [black, samples = 50, domain = -0.65 : -0.002]({x - 0.13},{1 + 2.5 * x});

\addplot [black, samples = 50, domain = -0.95 : 0.95]({-0.3},{x});

\addplot [black, samples = 50, domain = -0.116 : 0.125]({x - 0.24},{-8 * x});

\addplot [black, dashed, samples = 50, domain = -1 : 1]({0},{x});

\draw [fill] (1000, 2000) circle (0.05cm);

\end{axis}

\begin{axis}

\addplot [black, samples = 50, domain = -pi :  pi ]
(
    {cos(deg(x))}, {sin(deg(x))}
);

\addplot [red, samples = 50, domain =  - 5 * pi / 4 : pi / 4]
(
    {0.3 * cos(deg(x))}, {0.5 + 0.3 * sin(deg(x))}
);

\addplot [blue, samples = 50, domain =  -pi - pi / 3 - 0.145 : pi / 3 + 0.15]
(
    {0.6 * cos(deg(x))}, {0.15 + 0.6 * sin(deg(x))}
);

\addplot [black, samples = 50, domain = -0.65 : -0.002]({x - 0.13},{1 + 2.5 * x});

\addplot [black, samples = 50, domain = -0.95 : 0.95]({-0.3},{x});

\addplot [black, samples = 50, domain = -0.116 : 0.125]({x - 0.24},{-8 * x});

\addplot [black, dashed, samples = 50, domain = -1 : 1]({0},{x});

\draw [fill] (1000, 2000) circle (0.05cm);

\end{axis}

\end{tikzpicture}

\end{figure}
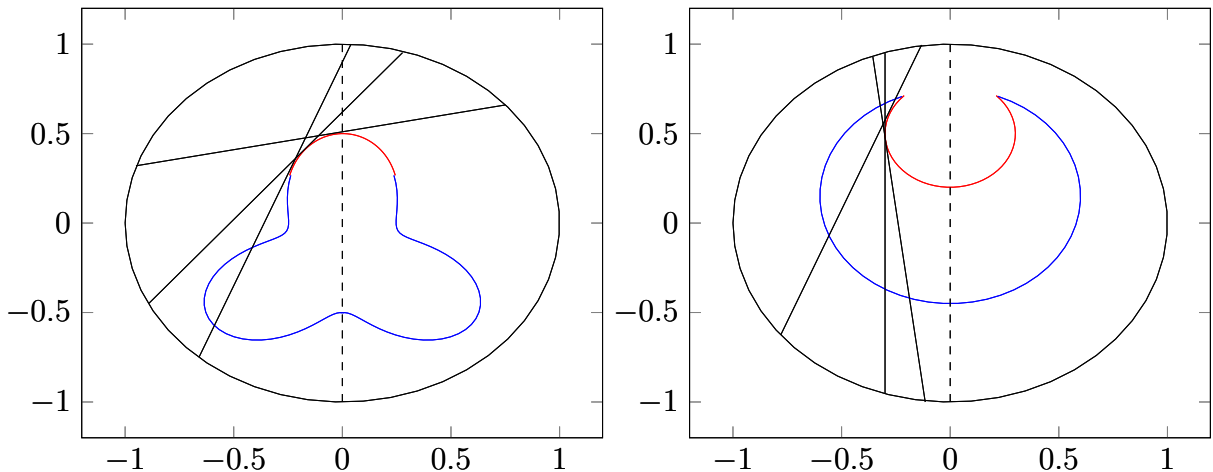

Hence, we are only left with the following possibility:\\

\hskip-0.6cm $\bullet$ Each tangent plane $\mathrm{H}_{i}$ does not intersect the segment $l$.\\

Using exactly the same argument as above we can prove that $f$ vanishes on $\mathrm{S}_{\mathrm{H}_{0}}^{-}$. Observe that $\mathrm{S}_{\mathrm{H}_{0}}^{-} = \Bbb S^{n}\cap\mathrm{H}_{0}^{-}$ where $\mathrm{H}_{0}^{-}$ is a half space which contains the cone $C_{\Sigma}$. Hence, if $\mathrm{H}_{\ast} = \mathbf{K}(\mathrm{H}_{0})$, where $\mathbf{K}$ is the map $(x^{\ast}, x_{n + 1})\mapsto (-x^{\ast}, x_{n + 1})$, then by simple geometrical considerations it can be easily proved that $\mathrm{H}_{0}^{-}$ contains the set $\mathrm{S}_{\mathrm{H}_{\ast}}^{+}$ which is obtained by intersecting $\Bbb S^{n}$ with the half space $\mathrm{H}_{\ast}^{+}$ which does not contain $C_{\Sigma}$. Thus $\mathrm{S}_{\mathrm{H}_{\ast}}^{+}\subset\mathrm{S}_{\mathrm{H}_{0}}^{-} $ and in particular $f$ vanishes on $\mathrm{S}_{\mathrm{H}_{\ast}}^{+}$. As in the previous case by rotating $\mathrm{S}_{\mathrm{H}_{\ast}}^{+}$ we can prove that $f$ is supported on $\Pi_{\Sigma}$. This finishes the proof of Theorem 3.1.

$\hskip14.25cm\square$\\

\end{document}